\documentclass[11pt]{amsart}
\usepackage[a4paper,margin=2.5cm]{geometry}
\usepackage{hyperref,mathtools,amsmath,amsthm,amssymb,amscd,amsfonts,color,tikz-cd,pgfplots,diagbox,dsfont}
\usepackage{booktabs}
\usepackage{adjustbox, xfrac, graphicx}
\usetikzlibrary{arrows}

\theoremstyle{plain}
\newtheorem{theorem}{Theorem}[section]
\newtheorem{proposition}[theorem]{Proposition}
\newtheorem{lemma}[theorem]{Lemma}
\newtheorem{corollary}[theorem]{Corollary}
\numberwithin{equation}{section}

\theoremstyle{definition}

\newtheorem{remark}[theorem]{Remark}
\newtheorem{example}[theorem]{Example}

\newtheorem{conjecture}[theorem]{Conjecture}

\newcommand{\cI}{\mathcal{I}}

\newcommand{\cF}{\mathcal{F}}

\newcommand{\cG}{\mathcal{G}}
\newcommand{\cZ}{\mathcal{Z}}

\newcommand{\cN}{\mathcal{N}}
\newcommand{\cU}{\mathcal{U}}
\newcommand{\cA}{\mathcal{A}}

\newcommand{\RP}{\mathbb{R}P}

\newcommand{\Q}{\mathbb{Q}}

\newcommand{\R}{\mathbb{R}}
\newcommand{\Z}{\mathbb{Z}}

\newcommand{\realbier}{M^\R(\Bier(K))}

\def\red#1{\textcolor{red}{#1}}

\DeclareMathOperator{\Ker}{Ker}

\DeclareMathOperator{\Tor}{Tor}

\DeclareMathOperator{\Lk}{Lk}

\DeclareMathOperator{\row}{row}

\DeclareMathOperator{\Bier}{Bier}

\begin{document}
\title[Full subcomplexes of Bier spheres]{Full subcomplexes of Bier spheres}

\author[S. Choi]{Suyoung Choi}
\address{Department of mathematics, Ajou University, 206, World cup-ro, Yeongtong-gu, Suwon 16499,  Republic of Korea}
\email{schoi@ajou.ac.kr}

\author[Y. Yoon]{Younghan Yoon}
\address{Department of mathematics, Ajou University, 206, World cup-ro, Yeongtong-gu, Suwon 16499,  Republic of Korea}
\email{younghan300@ajou.ac.kr}

\author[S. Yu]{Seonghyeon Yu}
\address{Department of mathematics, Ajou University, 206, World cup-ro, Yeongtong-gu, Suwon 16499,  Republic of Korea}
\email{yoosh0319@ajou.ac.kr}

\date{\today}


\thanks{This work was supported by the National Research Foundation of Korea Grant funded by the Korean Government (RS-2025-00521982).}

\begin{abstract}
Full subcomplexes of a simplicial complex encode essential structure for understanding the complex itself.
For a simplicial complex~$K$, possibly with a ghost vertex, the Bier sphere of~$K$ is a simplicial sphere obtained as the deleted join of~$K$ and its combinatorial Alexander dual.
In this paper, we determine the homotopy types of all full subcomplexes of Bier spheres.
As applications, we provide a formula for the bigraded Betti numbers of the Bier sphere of $K$ in terms of full subcomplexes of $K$, and we explicitly describe the cohomology of real toric manifolds associated with Bier spheres.
\end{abstract}
\maketitle
  \tableofcontents

\section{Introduction}
A \emph{Bier sphere} was introduced by Bier in~\cite{Bier1992} as a simplicial sphere, obtained in a simple way by taking the deleted join of a simplicial complex and its Alexander dual.
Let $m$ be a positive integer.
We denote the vertex sets by $[m] = \{1,2,\ldots,m\}$ and $[\bar{m}] = \{\bar{1},\bar{2},\ldots,\bar{m}\}$. For a subset $I \subset [m]$, we denote its corresponding set in $[\bar{m}]$ by $\bar{I} = \{\bar{i} \colon i \in I\}$.
Throughout this paper, we assume that $K$ is a simplicial complex on $[m]$ that is not the entire power set $2^{[m]}$, possibly including ghost vertices.
The (combinatorial) \emph{Alexander dual}~$\hat{K}$ of $K$ is a simplicial complex on~$[\bar{m}],$ defined as
$$  
    \hat{K} = \{\bar{\sigma} \subseteq [\bar{m}]\colon [m] \setminus \sigma \notin K \}.
$$  
The \emph{Bier sphere} of $K$, denoted by $\Bier(K)$, is a simplicial complex on~$[m] \sqcup [\bar{m}]$ and is defined as  
$$  
\Bier(K) = \{\sigma \sqcup \tau \colon \sigma \in K, \tau \in \hat{K}, \bar{\sigma} \cap \tau = \emptyset\}.  
$$
Here, $\sqcup$ denotes disjoint union.
By construction, $\Bier(K)$ is a simplicial sphere of dimension~$(m-2)$ with $2m$ vertices, though some of them may be ghost vertices, depending on the structure of~$K$.

Over the last few decades, Bier spheres have been studied from various viewpoints in \cite{Matousek2003, de_Longueville2004, Bjorner-Paffenholz-Sjostrand-Ziegler2005, Murai2011,Heudtlass-Katthan2012,Alessio-Gunnar-Nematbakhsh2019,Jevtic-Timotijevic-Zivaljevic2021,Ivan_Sergeev2024,Ivan-Ales2024,Ivan-Marinko-Rade2025}.
For instance, an alternative proof of the Van-Kampen-Flores theorem using the Bier sphere of a skeleton of a simplex was given in~\cite{Matousek2003}.

The study of Bier spheres is important also in toric geometry.
Limonchenko has shown in~\cite{Ivan-Marinko-Rade2025} that every Bier sphere supports a canonical complete nonsingular fan. 
Consequently, one may associate to any Bier sphere $\Bier(K)$ a toric manifold, denoted by $M(\Bier(K))$.
On the other hand, it is known that certain Bier spheres are non-polytopal~\cite{Bjorner-Paffenholz-Sjostrand-Ziegler2005}. 
In such cases, $M(\Bier(K))$ provides an alternative proof of the fact, first established in~\cite{Suyama2014}, that the underlying simplicial complex of a toric manifold need not be polytopal.

For a simplicial complex~$\Gamma$ on a finite set~$S$ and a subset $I$ of $S$, the \emph{full subcomplex} $\Gamma\vert_I$ of a simplicial complex $\Gamma$ with respect to $I$ is defined as
$$
\Gamma\vert_I = \{\sigma \in \Gamma : \sigma \subset I\}.
$$
Since the Bier sphere $\Bier(K)$ of $K$ is a simplicial complex on $[m]\sqcup [\bar{m}]$, its full subcomplexes $\Bier(K)\vert_{I\sqcup \bar{J}}$ are determined by a choice of the subsets $I$ and $J$ of $[m]$ such as
$$
\Bier(K)\vert_{I\sqcup \bar{J}} = \{\sigma\sqcup\bar{\tau}\in \Bier(K) : \sigma \subset I, \tau \subset J\}.
$$

In this paper, we study the full subcomplexes of a Bier sphere.
The \emph{link} of a simplex $\sigma$ in a simplicial complex $\Gamma$, denoted $\operatorname{Lk}(\sigma,\Gamma)$, is the subcomplex consisting of all simplices $\tau$ disjoint from $\sigma$ such that $\sigma \cup \tau \in \Gamma$, and the \emph{suspension} of $\Gamma$, denoted $\Sigma \Gamma$, is the join of $\Gamma$ with two new isolated vertices.
We denote the $k$-fold suspension by $\Sigma^k$, and set~$\Sigma^0\Gamma = \Gamma$.
The homotopy type of full subcomplexes of a Bier sphere is determined by the main theorem stated below, where the description is given in terms of links and suspensions defined above.

\begin{theorem}~\label{thm:main}
  Let $K$ be a simplicial complex on $[m]$, and let $I,J$ be subsets of $[m]$.
  Then the full subcomplex $\Bier(K)\vert_{I\sqcup\bar{J}}$ of $\Bier(K)$ with respect to $I\sqcup\bar{J}\subset [m]\sqcup[\bar{m}]$ is homotopy equivalent to
  \begin{enumerate}
    \item\label{thm:main1} \makebox[4cm][l]{the $(|I|-1)$-sphere,} if $I=J$, $I\in K$, and $\bar{I}\in\hat{K}$,
    \item\label{thm:main2} \makebox[4cm][l]{the $(|I|-2)$-sphere,} if $I=J$, $I\notin K$, and $\bar{I}\notin\hat{K}$,
    \item\label{thm:main3}
     \makebox[4cm][l]{$\Sigma^{|J|}\Lk(J,K\vert_I)$, } if $J\subsetneq I$ and $J\in K$,
    \item\label{thm:main4} \makebox[4cm][l]{$\Sigma^{|I|}\Lk(\bar{I},\hat{K}\vert_{\bar{J}})$,} if $J\supsetneq I$ and $\bar{I}\in\hat{K}$, 
    \item\label{thm:main5} \makebox[4cm][l]{a point,} otherwise.
  \end{enumerate}
  In particular, in the case of~\eqref{thm:main1}, $\Bier(K)\vert_{I\sqcup\bar{I}}$ is the boundary of an $|I|$-dimensional cross polytope.
\end{theorem}

In Section~\ref{sec:AKL}, we classify all types of full subcomplexes~$\Bier(K)_{I\sqcup \bar{I}}$ in terms of the faces of~$K$, thereby resolving~\eqref{thm:main1} and~\eqref{thm:main2}.
In Section~\ref{sec3}, we provide a cover of the full subcomplexes of Bier spheres consisting of the simplicial join of full subcomplexes of $K$ and $\hat{K}$:
$$
\Bier(K)_{I\sqcup \bar{J}}=
  \bigcup_{\sigma \in K\vert_{I\cap J}} K\vert_{(I\setminus I\cap J) \cup \sigma} \ast \hat{K}\vert_{\bar{J}\setminus \bar{\sigma}}
$$
for subsets $I,J$ of $[m]$.
Using this cover, we prove that any full subcomplex~$\Bier(K)_{I \sqcup \bar{J}}$ is contractible in the case where~$I$ and~$J$
are not comparable under inclusion.
In Section~\ref{sec:inclu}, we determine the homotopy types of $\Bier(K)\vert_{I\sqcup\bar{J}}$ by the link of~$J$ in~$K\vert_I$, as in \eqref{thm:main3} and \eqref{thm:main4}, in the case~$J \subsetneq I$.
Furthermore, in this case, we also prove that~$\Bier(K)\vert_{I\sqcup \bar{J}}$ is contractible whenever~$J\notin K$.

For a simplicial complex, its full subcomplexes reveal the underlying geometric and algebraic structure of the simplicial complex itself, and they have been studied from various viewpoints in topology, combinatorics, and commutative algebra~\cite{Reisner1976,Stanley1996book,Choi-Panov-Suh2010,Choi-Kim2011,Suciu2012,Choi-Park2015,Choi-Park2017_torsion,Choi-Park2017_multiplicative,Cai-Choi2021,Choi-Jang-Vallee2024, Choi-Yoon-Yu2024}.
For a topological pair~$(X,Y)$, the \emph{polyhedral product}~$(X,Y)^\Gamma$ of a simplicial complex~$\Gamma$ on a finite set $S$ with respect to $(X,Y)$ is
defined by
$$
    (X,Y)^\Gamma = \bigcup_{\sigma \in \Gamma}\{ (x_1,\ldots,x_{|S|}) \in X^{|S|} \colon x_i \in Y \text{ whenever } i \notin \sigma\}.
$$
In particular, the \emph{moment-angle complex}~$\cZ_\Gamma$ of~$\Gamma$ is given by
$(D^2,S^1)^\Gamma$, and the \emph{real moment-angle complex}~$\R\cZ_\Gamma$ by
$(D^1,S^0)^\Gamma$.
Full subcomplexes arise naturally in the general theory of polyhedral products and appear in many structural results including the Bahri--Bendersky--Cohen--Gitler decomposition introduced in \cite{BBCG2010}
$$
    \Sigma (X,Y)^\Gamma \simeq \bigvee_{I \subseteq S} \Sigma^{|I|+1} (X,Y)^{\Gamma\vert_I}.
$$

They play an even more prominent role in the study of real and complex moment-angle complexes.
It is well known~\cite{Buchstaber-Panov2015} that the cohomology groups $H^\ast(\cZ_\Gamma)$ and $H^\ast(\R\cZ_\Gamma)$ are  completely determined by the reduced cohomology of the full subcomplexes of $\Gamma$ and can be written as 
$$
    H^{-i,2j}(\cZ_\Gamma;\Z)    \cong
    \bigoplus_{\substack{W \subseteq S \\ |W| = j}}
    \widetilde{H}^{j-i-1}(\Gamma \vert_W; \Z), \quad \text{ and } \quad 
    H^{p}(\R\cZ_\Gamma;\Z)    \cong
    \bigoplus_{W \subseteq S}
    \widetilde{H}^{p-1}(\Gamma \vert_W; \Z).
$$
Therefore, from Theorem~\ref{thm:main}, we obtian a complete description of the integral cohomology groups of $\cZ(\Bier(K))$ and $\R\cZ(\Bier(K))$.
Note that $K$ appears as a full subcomplex of its Bier sphere $\Bier(K)$, as is also seen in Theorem~\ref{thm:main}\,\eqref{thm:main3}.
If $K$ has torsion in its cohomology, then the same torsion occurs in both $\cZ(\Bier(K))$ and $\R\cZ(\Bier(K))$.
While Bosio and Meersseman~\cite[Theorem~11.11]{Bosio-Meersseman2006} established the existence of (real) moment-angle complexes with arbitrary torsion, our observation supplies a concrete realization of this phenomenon through Bier spheres.

Furthermore, the cohomology ring of $\cZ_\Gamma$ is known to be isomorphic to the $\Tor$-algebra of the Stanley--Reisner ring of~$\Gamma$.
The rank of the free part of the $\Tor$-algebra in bidegree~$(-i,2j)$ is called the \emph{bigraded Betti number} $\beta^{-i,2j}(\Gamma)$ of~$\Gamma$.
In \emph{syzygy theory}, determining even a part of a Betti table is already a difficult and significant problem~\cite{Nagel-Pitteloud1994, Eisenbud2005, Han-Kwak2015}.

Attempts to compute the bigraded Betti numbers of Bier spheres naturally followed.
Heudtlass and Katth{\"a}n~\cite{Heudtlass-Katthan2012} computed $\beta^{-1,2j}(\Bier(K))$ and $\beta^{-2,2j}(\Bier(K))$ for arbitrary $K$, as well as the bigraded Betti numbers for Bier spheres of skeletons of simplices.
Beyond this special case, however, very little has been known.
Using Theorem~\ref{thm:main}, we can express the complete Betti table of $\Bier(K)$ in terms of the structure of~$K$, as summarized below.
Further details are presented in Section~\ref{sec5}.

Define two collections of subsets of $[m]$ as
\begin{align*}
   & \cF^+_k := \{\sigma\subseteq [m] : \left\vert \sigma \right\vert = \frac{k}{2}, \sigma\in K, \bar{\sigma} \in \hat{K} \}, \mbox{ and }\\
   & \cF^-_k := \{\tau\subseteq [m] : \left\vert \tau \right\vert = \frac{k}{2}, \tau\not\in K, \bar{\tau}\not\in \hat{K}\}. 
\end{align*}

\begin{theorem}~\label{thm:main_2}
  Let $K$ be a simplicial complex on $[m]$.
  For non-negative integers $i$ and $j$, the bigraded Betti number $\beta^{-i,2j}(\Bier(K))$ of $\Bier(K)$ in bidegree~$(-i,2j)$ is given by
  \begin{align*}
                              \delta_{2i,j} \left\vert \cF^+_j\right\vert + \delta_{2i-2,j}\left\vert \cF^-_j \right\vert
                              &
                              + \sum_{J \in K\vert_I} \widetilde{\beta}^{\left\vert I \right\vert-i-1}(\Lk(J, K\vert_I))
                             &
                              + \sum_{\bar{I} \in \hat{K}\vert_{\bar{J}}} \widetilde{\beta}^{\left\vert J \right\vert-i-1}(\Lk(\bar{I}, \hat{K}\vert_{\bar{J}}))
                              ,
  \end{align*}
  where~$I$ and~$J$ are distinct subsets of~$[m]$ such that $\left\vert I \right\vert + \left\vert J \right\vert = j$ and $\delta_{i,j}$ denotes Kronecker delta.
\end{theorem}

For a finite simplicial complex~$\Gamma$, a subgroup~$H$ of direct summands of~$\Z/2\Z$ acting freely on~$\R\cZ_\Gamma$ induces to the quotient space~$\R\cZ_\Gamma/H$.
This quotient space is referred to as a \emph{real toric space} over~$\Gamma$, which is a topological generalization~\cite{Choi-Kaji-Theriault2017} of a real toric manifold, namely the real locus of toric manifold.
It has been shown~\cite{Suciu2012,Choi-Park2017_torsion, Cai-Choi2021} that the rational Betti numbers of a real toric space can be computed from the reduced Betti numbers of full subcomplexes of~$\Gamma$.
With the homotopy types of all full subcomplexes of Bier spheres at hand, this enables us to compute the rational Betti numbers of any real toric space over Bier spheres.

Recall that each Bier sphere~$\Bier(K)$ gives rise to a toric manifold~$M(\Bier(K))$, and hence to a real toric manifold~$\realbier$.
In Section~\ref{sec:topo}, we apply Theorem~\ref{thm:main} to $\realbier$ for its cohomology.
We collect the subsets of $[m]$ that are either contained in both $K$ and $\hat{K}$ or in neither, and denote them as follows  
\begin{align*}  
    \cI_{2k}(K) &:= \{ I \subset [m] \colon |I|=2k, I \in K, \text{ and } \bar{I} \in \hat{K} \}, \text{ and}\\    
    \cI_{2k-1}(K) &:= \{ I \subset [m] \colon |I|=2k, I \not\in K, \text{ and } \bar{I} \not\in \hat{K} \}.  
\end{align*}  
\begin{theorem}\label{main2}
    Let $K$ be a simplicial complex on~$[m]$ and $\realbier$ the real toric manifold associated with $\Bier(K)$.
    For each $0 \leq i \leq m-1$, the $i$th (rational) Betti number of~$\realbier$ is $\left\vert \cI_i(K) \right\vert$.
    Moreover, the integral cohomology of $\realbier$ is $p$-torsion free for all $p \geq 3$.
\end{theorem}

Additionally, in Section~\ref{sec:topo}, we also discuss the multiplicative structure of the cohomology ring of~$\realbier$ based on~\cite{Choi-Park2017_multiplicative}.

\section{Preliminaries}
\subsection{Bier spheres}
Let $S$ be a finite set.
A finite collection $\Gamma$ of subsets of $S$ is an (abstract) \emph{simplicial complex} on $S$ if, for each $\sigma \in \Gamma$, every subset of $\sigma$ is also an element of $\Gamma$.
For each $1 \leq \ell \leq \left\vert S \right\vert$, an element $\{i_1,\ldots,i_\ell\} \subset S$ of $\Gamma$ is called a \emph{simplex} or a \emph{face} of $\Gamma$ of dimension~$\ell-1$.
Each one-point set $\{v\} \subset S$ is called a \emph{vertex} of $\Gamma$ and will simply be denoted by $v$.
If $v \in S$ is not a face of~$\Gamma$, then it is called a \emph{ghost vertex} of~$\Gamma$.
A maximal face of~$\Gamma$ under inclusion is called a \emph{facet} of~$\Gamma$.
The dimension $\dim \Gamma$ of $\Gamma$ is the largest dimension among its facets.
A simplicial complex $\Gamma$ is said to be \emph{pure}, if all its facets have the same dimension.
A subcollection~$\Gamma' \subset \Gamma$ that is itself a simplicial complex is called a \emph{subcomplex} of $\Gamma$.
For any subset $I \subset S$, the \emph{full subcomplex} $\Gamma\vert_I$ of $\Gamma$ with respect to $I$ is given by
$$
    \Gamma\vert_I= \{\sigma \in \Gamma \colon \sigma \subseteq I\}.
$$
For each pair of simplicial complexes $\Gamma_1$ and $\Gamma_2$, we say that $\Gamma_1$ is \emph{(simplicial) isomorphic} to~$\Gamma_2$, denoted by $\Gamma_1 \cong \Gamma_2$, if there exists a bijection between their vertex sets that induces a bijection on their simplices.

For simplicial complexes $\Gamma_1$ on a vertex set $S_1$ and $\Gamma_2$ on a vertex set $S_2$, the \emph{simplicial join}~$\Gamma_1 \ast \Gamma_2$ of~$\Gamma_1$ and~$\Gamma_2$ is the simplicial complex on~$S_1\cup S_2$, defined as
$$
    \Gamma_1\ast \Gamma_2 = \{\sigma\cup \tau : \sigma \in \Gamma_1, \tau \in \Gamma_2\}.
$$

For any set~$I$ of vertices, we denote its corresponding vertex set by $\bar{I} = \{\bar{i} \colon i \in I\}$.
For a simplicial complex~$\Gamma$ on a finite set~$S$, the (combinatorial) \emph{Alexander dual} $\hat{\Gamma}$ of $\Gamma$ is defined as a simplicial complex on $\bar{S}$ given by
$$
  \hat{\Gamma} = \{\bar{\sigma} \subseteq \bar{S} \colon \sigma \subseteq S \text{ and } S \setminus \sigma \notin \Gamma\}.
$$ 
One can see that the Alexander dual of $\hat{\Gamma}$ is $\Gamma$ itself, as we naturally regard $\bar{\bar{S}} = S$.
Given simplicial complexes $\Gamma_1$ on $S$ and $\Gamma_2$ on $\bar{S}$, the \emph{deleted join} of $\Gamma_1$ and $\Gamma_2$, denoted by $\Gamma_1 \ast_{\Delta} \Gamma_2$, is defined as
$$
\Gamma_1 \ast_{\Delta} \Gamma_2=\{\sigma \sqcup \tau \mid \sigma \in \Gamma_1, \tau \in \Gamma_2, \bar{\sigma} \cap \tau = \emptyset\}.
$$

Let us denote $[m] = \{1,\ldots,m\}$ and $[\bar{m}] = \{\bar{1},\ldots,\bar{m}\}$.
For a simplicial complex~$K$ on~$[m]$, the \emph{Bier sphere} of $K$, denoted by $\Bier(K)$, is the simplicial complex on $[m] \sqcup [\bar{m}]$ defined as the deleted join~$K \ast_{\Delta} \hat{K}$ of~$K$ and~$\hat{K}$, introduced in~\cite{Bier1992}.
As an example of a Bier sphere, see Figure~~\ref{sec2:figure1}.
Moreover, it can be immediately obtained that every Bier sphere of a simplicial complex on~$[m]$ is an $(m-2)$-dimensional simplicial sphere.
According to~\cite{Bjorner-Paffenholz-Sjostrand-Ziegler2005}, every Bier sphere is shellable, but in general, it is not necessarily polytopal. 
\begin{figure}
  \begin{tikzpicture}
        \fill[fill=gray!50] (-3,0)--(-4,-2)--(-2,-2);
        \node[left] at (-4,1) {\large{$K$}};
        \foreach \x/\y/\num in {-3/0/1, -2/-2/2, -4/-2/3} {
            \fill (\x, \y) circle (1.5pt);
        }
        \fill (-1,0) circle (1.5pt);
        \node at (-3,0.3) {1};
        \node at (-4,-2.3) {2};
        \node at (-2,-2.3) {3};
        \node at (-1,0.3) {4};
        
        \draw (-3,0) -- (-4,-2);
        \draw (-3,0) -- (-2,-2);
        \draw (-2,-2) -- (-4,-2);

        \node[left] at (1,1) {\large{$\hat{K}$}};
        \foreach \x/\y/\num in {2/0/1, 3/-2/3, 1/-2/4} {
            \fill (\x, \y) circle (1.5pt);
        }
        \draw (4,0) circle (1.5pt);
        \node at (2,0.3) {$\bar{1}$};
        \node at (1,-2.3) {$\bar{2}$};
        \node at (3,-2.3) {$\bar{3}$};
        \node at (4,0.3) {$\bar{4}$};

        \draw (2,0) -- (3,-2);
        \draw (2,0) -- (1,-2);
        \draw (1,-2) -- (3,-2);
        
        \node[left] at (6.6,1) {\large{$\Bier(K)$}};
        \filldraw[gray!20] (7.5, 0.5) -- (6.25, -1) -- (8, -1) -- cycle;
    \filldraw[gray!20] (7.5, 0.5) -- (8, -1) -- (8.75, -0.5) -- cycle;
    \filldraw[gray!20] (7.5, -2) -- (6.25, -1) -- (8, -1) -- cycle;
    \filldraw[gray!20] (7.5, -2) -- (8, -1) -- (8.5, -1.5) -- cycle;
    \filldraw[gray!20] (8.75, -0.5) -- (8, -1) -- (8.5, -1.5) -- cycle;
    \filldraw[gray!20] (7.5, 0.5) -- (6.25, -1) -- (7, -0.5) -- cycle;
    \filldraw[gray!20] (7.5, 0.5) -- (8.75, -0.5) -- (7, -0.5) -- cycle;
    \filldraw[gray!20] (7.5, -2) -- (7, -0.5) -- (6.25, -1) -- cycle;
    \filldraw[gray!20] (7.5, -2) -- (7, -0.5) -- (8.75, -0.5) -- cycle;
        \fill (7, -0.5) circle (1.5pt);
        \fill (6.25, -1) circle (1.5pt);
        \fill (8, -1) circle (1.5pt);
        \fill (8.75, -0.5) circle (1.5pt);
        \fill (7.5, 0.5) circle (1.5pt);
        \fill (7.5, -2) circle (1.5pt);
        \fill (8.5, -1.5) circle (1.5pt);
        \draw (6.5, 0) circle (1.5pt);

        \node at (7.5, 0.8) {1};
        \node at (7.2,-0.3) {2};
        \node at (6,-1) {3};
        \node at (8.7, -1.6) {4};
        \node at (7.5,-2.3) {$\bar{1}$};
        \node at (8.1, -0.7) {$\bar{2}$};
        \node at (8.9,-0.3) {$\bar{3}$};
        \node at (6.5, 0.3) {$\bar{4}$};

        \draw[dotted] (7.5, 0.5) -- (7, -0.5);
        \draw (7.5, 0.5) -- (6.25, -1);
        \draw (7.5, 0.5) -- (8,-1);
        \draw (7.5, 0.5) -- (8.75, -0.5);
        
        \draw[dotted] (7, -0.5) -- (6.25, -1);
        \draw (6.25, -1) -- (8, -1);
        \draw (8, -1) -- (8.75, -0.5);
        \draw[dotted] (8.75, -0.5) -- (7, -0.5);
        
        \draw[dotted] (7.5, -2) -- (7, -0.5);
        \draw (7.5, -2) -- (6.25, -1);
        \draw (7.5, -2) -- (8,-1);
        \draw[dotted] (7.5, -2) -- (8.75, -0.5);
        
        \draw (8.5, -1.5) -- (8.75, -0.5);
        \draw (8.5, -1.5) -- (8, -1);
        \draw (8.5, -1.5) -- (7.5, -2);
    \end{tikzpicture}
    \caption{A Bier sphere}
    \label{sec2:figure1}
    \end{figure}

For a simplicial complex $K$ on $[m]$, since the Bier sphere $\Bier(K)$ of $K$ is the simplicial complex on $[m]\sqcup[\bar{m}]$, its full subcomplex~$\Bier(K)\vert_{I\sqcup\bar{J}}$ is depending on the choice of the subsets~$I$ and~$J$ of~$[m]$.
Note that the full subcomplex~$\Bier(K)\vert_{I \sqcup \bar{J}}$ can be expressed by the deleted join of~$K\vert_{I}$ and~$\hat{K}\vert_{\bar{J}}$, namely
\begin{equation*}\label{eq:fullsubBier}
  \Bier(K)\vert_{I \sqcup \bar{J}} = K\vert_{I} \ast_{\Delta} \hat{K}\vert_{\bar{J}}.
\end{equation*}

\subsection{Real toric spaces}
Let $\Gamma$ be a simplicial complex on~$[m]$.
The \emph{real moment-angle complex}~$\R\cZ_{\Gamma}$ of $\Gamma$ is defined as
$$
\R\cZ_{\Gamma} = \bigcup_{\sigma \in \Gamma}\bigl\{ (x_1,\ldots,x_m) \in (D^1)^m \colon x_i \in S^0 \text{ if } i \notin \sigma \bigl\},
$$
where $D^1 = [-1,1]$ is a closed interval and $S^0 = \{-1,1\}$ is the boundary of $D^1$.
Given a linear map~$\lambda \colon (\Z/2\Z)^m \to (\Z/2\Z)^n$, where $n \leq m$, the sign action of $(\Z/2\Z)^m$ on $(D^1)^m$ induces an action of $\ker \lambda \subset (\Z/2\Z)^m$ on~$\R\cZ_{\Gamma}$.
The \emph{real toric space}~$M^\R(\Gamma,\lambda)$ associated with~$\Gamma$ and~$\lambda$ is defined as the quotient space
$$
M^\R(\Gamma,\lambda) = \R\cZ_{\Gamma}/\ker \lambda.
$$
It is known~\cite{Choi-Kaji-Theriault2017} that the action of $\ker\lambda$ on $\R\cZ_{\Gamma}$ is free if and only if the images
$$
\lambda(e_{i_1}),\ldots,\lambda(e_{i_r})
$$
are linearly independent in $(\Z/2\Z)^n$ for each simplex~$\{i_1,\ldots,i_r\}$ of~$\Gamma$, where $e_i$ is the $i$th standard vector of $(\Z/2\Z)^m$.
In this case, the linear map~$\lambda$ is called a (mod $2$) \emph{characteristic map} over $\Gamma$.
Note that $\dim \Gamma \leq n-1$.
We set $\dim \Gamma = n-1$.

Let $M^\R$ be a real toric space associated with a pair of a simplicial complex~$\Gamma$ on $[m]$ and a (mod $2$) characteristic map~$\lambda \colon (\Z/2\Z)^m \to (\Z/2\Z)^n$ over~$\Gamma$.
If $\Gamma$ is polytopal, then the associated real toric space~$M^\R$ is known as a \emph{small cover}~\cite{Davis-Januszkiewicz1991}, and if $\Gamma$ is star-shaped, then $M^\R$ is called a \emph{real topological toric manifold}~\cite{Ishida-Fukukawa-Masuda2013}.
A (mod $2$) characteristic map~$\lambda$ can be regarded as a (mod $2$) $n \times m$ matrix~$\Lambda$, called a (mod $2$) \emph{characteristic matrix}, where each column is $\lambda(v)$ for~$v \in (\Z/2\Z)^m$.
We denote by~$\row \Lambda$ the row space of this matrix.
Each element~$\omega \in \row \Lambda$ can be identified with a subset of $[m]$ via the standard correlation between the power set of $[m]$ and $(\Z/2\Z)^m$.
Then, $\omega$ corresponds to the full subcomplex~$\Gamma_\omega$ of $\Gamma$ with respect to~$\omega$.

In this paper, for a topological space~$X$, we denote the ordinary cohomology of~$X$ with coefficients in a ring~$R$ by~$H^{\ast}(X;R)$, and the reduced cohomology by~$\widetilde{H}^{\ast}(X;R)$.
For simplicity, when $R = \Q$, we occasionally omit the coefficient notation if there is no confusion.
Note that the formula of the cohomology ring of $M^\R$ with the $\Z/2\Z$ coefficient was established in \cite{Jurkiewicz1985, Davis-Januszkiewicz1991}. 
The cohomology of $M^\R$ with other coefficients can be obtained from the pair~$(\Gamma,\Lambda)$ as stated in the following two theorems.

\begin{theorem}\cite{Suciu2012, Choi-Park2017_torsion, Cai-Choi2021}\label{CaiChoi}
    For the integral cohomology group~$H^\ast(M^\R;\Z)$ of~$M^\R$,
    \begin{enumerate}
      \item the free part (respectively, for an odd prime $p$, the $p^k$-torsion part) of $H^i(M^\R;\Z)$ coincides with that of $\underset{\omega \in \row\Lambda}{\bigoplus}\widetilde{H}^{i-1}(\Gamma_\omega;\Z)$.
      \item If $\Gamma$ is shellable and pure, then the $2^{k+1}$-torsion part of $H^i(M^\R;\Z)$ coincides with the $2^{k}$-torsion part of~$\underset{\omega \in \row\Lambda}{\bigoplus}\widetilde{H}^{i-1}(\Gamma_\omega;\Z)$.
    \end{enumerate}
\end{theorem}

\begin{theorem}\cite{Choi-Park2017_multiplicative}\label{ChoiPark}
  There is a graded $\Q$-algebra isomorphism
$$
  H^{\ast}(M^\R;\Q) \cong \underset{\omega \in \row\Lambda}\bigoplus \widetilde{H}^{\ast-1}(\Gamma_\omega;\Q).
$$
The multiplicative structure on $\underset{\omega \in \row\Lambda}\bigoplus \widetilde{H}^{\ast}(\Gamma_\omega;\Q)$ is defined via the canonical maps
$$
  \widetilde{H}^{k-1}(\Gamma_{\omega_1}) \otimes \widetilde{H}^{\ell-1}(\Gamma_{\omega_2}) \to \widetilde{H}^{k+\ell-1}(\Gamma_{\omega_1+\omega_2})
$$
induced by simplicial maps $\Gamma_{\omega_1+\omega_2} \to \Gamma_{\omega_1} \ast \Gamma_{\omega_2}$, where $\ast$ denotes the simplicial join.
\end{theorem}

For each integer~$m \geq 2$, let $\Lambda_{m}$ be the (mod $2$) $(m-1) \times 2m$ matrix defined as
\begin{equation}\label{matrix}
  \Lambda_{m} =
\left(
\begin{array}{ccccccccc}
1 & \cdots & m-1 & m & \bar{1} & \cdots & \overline{m-1} & \bar{m}\\
\hline
1 & \cdots & 0 & 1 & 1 & \cdots & 0 & 1\\
\vdots & \ddots & \vdots & 1 & \vdots & \ddots & \vdots & 1\\
0 & \cdots & 1 & 1 & 0 & \cdots & 1 & 1\\
\end{array}
\right),
\end{equation}
where the first row indicates the labels of the vertices.
In other words, for each $1 \leq i \leq m$, the columns of $\Lambda_{m}$ indexed by $i$ and $\bar{i}$ both  equal $e_i$, where $e_i$ is the $i$th standard vector of~$(\Z/2\Z)^{m-1}$ and $e_{m} = e_1 + \cdots + e_{m-1}$.
For any simplicial complex~$\Gamma$ on~$[m]\sqcup [\bar{m}]$, one can see that~$\Lambda_m$ is a characteristic matrix over~$\Gamma$ if and only if the vertices~$i$ and $\bar{i}$ are not adjacent in~$\Gamma$ for each~$1 \leq i \leq m$.
Therefore, for any simplicial complex~$K$ on~$[m]$, $\Lambda_m$ forms a characteristic matrix over~$\Bier(K)$.
Refer~\cite{Ivan_Sergeev2024} for more details.
Since the characteristic matrix~$\Lambda_m$ induces a smooth and complete fan~\cite{Ivan-Marinko-Rade2025}, the corresponding real toric space~$\realbier := M^\R(\Bier(K),\Lambda_m)$ is a real toric manifold.

Let $\omega_i$ be the element of the row space~$\row \Lambda_m$ of~$\Lambda_m$ corresponding to the $i$th row of~$\Lambda_m$ for~$1 \leq i \leq m-1$.
Then, each element~$\omega = \omega_{i_1} +\cdots + \omega_{i_r}$ of~$\row \Lambda_m$ can be regarded as a subset~$I_\omega \subset [m]$ that has an even cardinality as follows
$$
I_\omega=\begin{cases}
           \{i_1,\ldots,i_{r}\}, & \text{if $r$ is even}, \\
           \{i_1,\ldots,i_{r},m\}, & \text{if $r$ is odd}.
         \end{cases}
$$
Conversely, for each even subset $I$ of $[m]$, there is an element $\omega \in \row \Lambda_m$ corresponding to~$I$.
Therefore, an element $\omega \in \row\Lambda$ corresponds to an even subset $I_\omega$ of $[m]$.
Since the $i$th and $\bar{i}$th columns are equal in \eqref{matrix}, one can check that the full subcomplexes appeared in Theorem~\ref{CaiChoi} are expressed as 
\begin{equation} \label{eq:row-2subset}
    \Bier(K)_\omega = \Bier(K)\vert_{I_{\omega} \sqcup \bar{I}_{\omega}}.
\end{equation}

On the other hand, the summation of two elements~$\omega$ and $\omega'$ in $\row \Lambda_m$ corresponds to the symmetric difference~$I_{\omega} \triangle I_{\omega'}$ of~$I_{\omega}$ and~$I_{\omega'}$ as
$$
I_{\omega+\omega'} = I_{\omega} \triangle I_{\omega'} := (I_{\omega} \cup I_{\omega'}) \setminus (I_{\omega} \cap I_{\omega'}).
$$
Note that the symmetric difference~$I_{\omega} \triangle I_{\omega'}$ of~$I_{\omega}$ and $I_{\omega'}$ always has an even cardinality.

\section{$\Bier(K)\vert_{I\sqcup\bar{J}}$ when $I=J$}\label{sec:AKL}

Consider a finite vertex set $S$ and its corresponding set $\bar{S}$.
Let $K$ and $L$ be simplicial complexes on $S$ and $\bar{S}$, respectively.
Recall that the Alexander dual~$\hat{L}$ of $L$ is defined on $S$, whereas the Alexander dual~$\hat{K}$ of $K$ is on $\bar{S}$.

Here, we define the collection $\cA_{(K,L)}$, which plays a key role in studying the full subcomplex of a Bier sphere for $I=J$ case.
The collection~$\cA_{(K,L)}$ consists of subsets of $S \sqcup \{v_S\}$ as
$$
 \cA_{(K,L)} = K \cup \{\sigma \cup \{v_S\} \colon \sigma \in \hat{L}\},
$$
where $v_S$ is a vertex not contained in $S$.
By regarding~$S$ as the corresponding set of~$\bar{S}$, the collection $\cA_{(L,K)}$ on $\bar{S} \sqcup \{v_{\bar{S}}\}$ is similarly obtained by
$$
    \cA_{(L,K)} = L \cup \{\tau \cup \{v_{\bar{S}}\} \colon \tau \in \hat{K}\}.
$$
Note that $v_S$ is a ghost vertex of $\cA_{(K,L)}$ if and only if $L$ is the power set~$2^{\bar{S}}$ of~$\bar{S}$. 
In general, a collection $\cA_{(K,L)}$ is not a simplicial complex.

\begin{proposition}~\label{sec3:prop1}
    Let $K$ and $L$ be simplicial complexes on $S$ and $\bar{S}$, respectively.
If the Alexander dual $\hat{L}$ of $L$ is a subcomplex of~$K$, then $\cA_{(K,L)}$ forms a simplicial complex on $S \sqcup \{v_S\}$.
          In this case, $K$ is the full subcomplex of $\cA_{(K,L)}$ with respect to $S$.    
\end{proposition}

\begin{proof}
    Since every element of~$K$ is also included in~$\cA_{(K,L)}$, it suffices to show that $\cA_{(K,L)}$ forms a simplicial complex.
    
    Let $\sigma$ be an element of~$\cA_{(K,L)}$.
    If $\sigma$ is in $K$, then so is every subset of~$\sigma$, implying that they all belong to $\cA_{(K,L)}$.
    Now, assume that $\sigma \notin K$. 
    Then, $v_S \in \sigma$ and $\sigma \setminus \{v_S\} \in \hat{L}$.
    Consider any subset~$\eta$ of $\sigma$.
    If $v_S \in \eta$, then since $\hat{L}$ is a simplicial complex, $\eta \setminus \{v_S\}$ belongs to $\hat{L}$.
    Otherwise, if $v_S \notin \eta$, then $\eta$ is in $\hat{L} \subset K$ by assumption.
    In both cases, we conclude that $\eta \in \cA_{(K,L)}$, as desired.
\end{proof}

\begin{example}\label{sec3:example}
Consider simplicial complexes
\begin{align*}
  K & =\bigl\{\emptyset, \{1\}, \{2\},\{3\},\{4\},\{1,2\},\{2,3\},\{1,3\},\{1,2,3\}\bigl\}, \text{ and} \\
  L & =\bigl\{\emptyset, \{\bar{1}\}, \{\bar{2}\},\{\bar{3}\},\{\bar{4}\},\{\bar{1},\bar{2}\},\{\bar{2},\bar{3}\},\{\bar{1},\bar{3}\},\{\bar{2},\bar{4}\},\{\bar{3},\bar{4}\}\bigl\}
\end{align*}
on $S=\{1,2,3,4\}$ and $\bar{S}=\{\bar{1},\bar{2},\bar{3},\bar{4}\}$, respectively.
We have that 
$$\hat{K}=\bigl\{\emptyset, \{\bar{1}\}, \{\bar{2}\}, \{\bar{3}\}, \{\bar{1},\bar{2}\}, \{\bar{2},\bar{3}\}, \{\bar{1},\bar{3}\}\bigl\}  \text{ and } \hat{L}=\bigl\{\emptyset, \{1\}, \{2\},\{3\},\{4\},\{2,3\}\bigl\}.
$$
It follows that
\begin{align*}
  \cA_{(K,L)} &=K \cup \bigl\{\{v_S\},\{1,v_S\},\{2,v_S\},\{3,v_S\},\{4,v_S\},\{2,3,v_S\}\bigl\}, \text{ and }\\
  \cA_{(L,K)} &= L \cup \bigl\{\{v_{\bar{S}}\}, \{\bar{1}, v_{\bar{S}}\}, \{\bar{2}, v_{\bar{S}}\}, \{\bar{3}, v_{\bar{S}}\}, \{\bar{1},\bar{2},v_{\bar{S}}\}, \{\bar{2},\bar{3},v_{\bar{S}}\}, \{\bar{1},\bar{3},v_{\bar{S}}\} \bigl\}.
\end{align*}
It is straightforward to check that $K$ contains $\hat{L}$, and $L$ contains $\hat{K}$ as subcomplexes.
Indeed, both $\cA_{(K,L)}$ and $\cA_{(L,K)}$ are simplicial complexes.
\begin{figure}
\centering
  \begin{tikzpicture}[scale=0.95]
        \fill[fill=gray!50] (-4,0)--(-5,-2)--(-3,-2);
        \node[left] at (-5,1) {\large{$K$}};
        \foreach \x/\y/\num in {-4/0/1,-2.5/0/1, -3/-2/2, -5/-2/3} {
            \fill (\x, \y) circle (1.5pt);
        }
        \node at (-4,0.3) {1};
        \node at (-5,-2.3) {2};
        \node at (-3,-2.3) {3};
        \node at (-2.5,0.3) {4};
        
        \draw (-4,0) -- (-5,-2);
        \draw (-4,0) -- (-3,-2);
        \draw (-3,-2) -- (-5,-2);
        \node[left] at (-1,1) {\large{$L$}};
        \foreach \x/\y/\num in {0/0/1,1.5/0/2, 1/-2/3, -1/-2/4} {
            \fill (\x, \y) circle (1.5pt);
        }
        \node at (0,0.3) {$\bar{1}$};
        \node at (-1,-2.3) {$\bar{2}$};
        \node at (1,-2.3) {$\bar{3}$};
        \node at (1.5,0.3) {$\bar{4}$};

        \draw (0,0) -- (1,-2);
        \draw (0,0) -- (-1,-2);
        \draw (-1,-2) -- (1,-2);
        \draw (1.5,0) -- (1,-2);
        \draw (1.5,0) -- (-1,-2);

        \fill[fill=gray!30] (4.5,-2)--(3,-2)--(5,-0.5);
        \fill[fill=gray!50] (3.75,0.5)--(3,-2)--(4.5,-2);
        \node[left] at (3.85,1) {\large{$\cA_{(K,L)}$}};
        \foreach \x/\y/\num in {3.75/0.5/5, 5/-0.5/1, 4.5/-2/3, 3/-2/2, 6.25/-0.5/4} {
            \fill (\x, \y) circle (1.5pt);
        }
        \node at (5.25,-0.5) {1};
        \node at (3,-2.3) {2};
        \node at (4.5,-2.3) {3};
        \node at (6.5, -0.5) {4};
        \node at (4.05,0.65) {$v_S$};

        \draw (5,-0.5) -- (4.24,-1.07);
        \draw[dotted] (4.24,-1.07) -- (3,-2);        
        \draw (3.75,0.5) -- (3,-2);
        \draw (3.75,0.5) -- (5,-0.5);
        \draw (3.75,0.5) -- (4.5,-2);
        \draw (5,-0.5) -- (4.5,-2);
        \draw (3,-2) -- (4.5,-2);
        \draw (6.25,-0.5) -- (3.75,0.5);

        \fill[fill=gray!50] (9.25,-0.5)--(8,0.5)--(7.25,-2);
        \fill[fill=gray!30] (9.25,-0.5)--(7.25,-2)--(8.75,-2);
        \node[left] at (8.1,1) {\large{$\cA_{(L,K)}$}};
        \foreach \x/\y/\num in {9.25/-0.5/1, 7.25/-2/2, 8.75/-2/3, 10.5/-0.5/4, 8/0.5/1} {
            \fill (\x, \y) circle (1.5pt);
        }
        \node at (9.5,-0.5) {$\bar{1}$};
        \node at (7.2,-2.3) {$\bar{2}$};
        \node at (8.8,-2.3) {$\bar{3}$};
        \node at (10.75, -0.5) {$\bar{4}$};
        \node at (8.3, 0.65) {$v_{\bar{S}}$};
        
        \draw[dotted] (9.25,-0.5) -- (7.25,-2);
        \draw (9.25,-0.5) -- (8.75,-2);
        \draw (7.25,-2) -- (8.75,-2);
        
        \draw (10.5, -0.5) -- (9.023, -1.182);
        \draw[dotted] (9.023, -1.182) -- (7.25,-2);
        \draw (10.5, -0.5) -- (8.75, -2);        
        
        \draw (8,0.5) -- (9.25,-0.5);
        \draw (8,0.5) -- (7.25,-2);
        \draw (8,0.5) -- (8.75,-2);
    \end{tikzpicture}
    \caption{Simplicial complexes $\cA_{(K,L)}$ and $\cA_{(L,K)}$}
    \label{sec3:figure1}
    \end{figure}
Refer Figure~\ref{sec3:figure1}.
\end{example}

It should be noted that, in Example~\ref{sec3:example}, the Alexander dual of~$\cA_{(K,L)}$ is $\cA_{(L,K)}$.
This phenomenon can be generalized as follows.

\begin{lemma}\label{sec3:lemma1}
Let $K$ and $L$ be simplicial complexes on $S$ and $\bar{S}$, respectively.
Assume that $\hat{K}$ is a subcomplex of $L$ and that $\hat{L}$ is a subcomplex of $K$.
Then, the Alexander dual~$\hat{\cA}_{(K,L)}$ of~$\cA_{(K,L)}$ coincides with $\cA_{(L,K)}$, identifying~$\overline{v_S}$ with~$v_{\bar{S}}$.
\end{lemma}

\begin{proof}

    We now show that $\hat{\cA}_{(K,L)} \subset \cA_{(L,K)}$.    
    Consider a subset $\sigma \subset S \sqcup \{v_S\}$ such that $\bar{\sigma} \in \hat{\cA}_{(K,L)}$, that is,
    $$
      (S \sqcup \{v_S\}) \setminus \sigma \notin \cA_{(K,L)}.
    $$

    To prove that $\bar{\sigma} \in \cA_{(L,K)}$, we distinguish two cases; $v_S \in \sigma$ and $v_S \notin \sigma$.
    First, assume that~$v_S \in \sigma$.
    Then, we have
    $$
    S \setminus (\sigma \cap S) = (S \sqcup \{v_S\}) \setminus \sigma \notin K.
    $$
    This implies that $\overline{\sigma \cap S} = \bar{\sigma} \cap \bar{S}$ is a simplex in the Alexander dual~$\hat{K}$ of~$K$.
    Since $v_{\bar{S}}$ must be contained in $\bar{\sigma}$, we obtain $\bar{\sigma} = (\bar{\sigma} \cap \bar{S}) \cup \{v_{\bar{S}}\}$.
    Therefore, $\bar{\sigma} \in \cA_{(L,K)}$.
    Now, consider the case where $v_S \notin \sigma$.
    In this case, we obtain that
    $$
        (S \setminus \sigma) \sqcup \{v_S\} = (S \sqcup \{v_S\}) \setminus \sigma \notin \cA_{(K,L)}.
    $$
    It follows that $S \setminus \sigma \notin \hat{L}$, and thus, $\bar{\sigma} \in L \subset \cA_{(L,K)}$, as desired.
    This completes the proof of~$\hat{\cA}_{(K,L)} \subset \cA_{(L,K)}$. 

    We next show the reverse inclusion  $\hat{\cA}_{(K,L)} \supset \cA_{(L,K)}$.
    Consider a subset  $\tau \subseteq \bar{S} \sqcup \{v_{\bar{S}}\}$ such that $\tau \in \cA_{(L,K)}$.    
    As before, we proceed with the proof by investigating two cases; either $v_{\bar{S}} \in \tau$ or $v_{\bar{S}} \notin \tau$.
    If $v_{\bar{S}} \in \tau$, then it immediately follows that $\tau \notin L$, and consequently, $\tau \cap \bar{S} \in \hat{K}$.
    Therefore, we have    
    \begin{equation}\label{in_proof1}
     (S \sqcup \{v_S\}) \setminus \bar{\tau} = S \setminus (\bar{\tau} \cap S) \notin K.
    \end{equation}
    Moreover, since $v_S \in \bar{\tau}$, we deduce that
    \begin{equation}\label{in_proof2}
      (S \sqcup \{v_S\}) \setminus \bar{\tau} \notin \{\sigma \sqcup \{v_S\} \colon \sigma \in \hat{L}\}.
    \end{equation}
    By~\eqref{in_proof1} and \eqref{in_proof2}, we conclude that $\tau$ is a simplex in~$\hat{\cA}_{(K,L)}$.
       
    Otherwise, if $v_{\bar{S}} \notin \tau$, then $\tau \in L$, which implies that $S \setminus \bar{\tau} \notin \hat{L}$.
    Since $v_S \notin \bar{\tau}$, 
    $$
    (S \sqcup \{v_S\}) \setminus \bar{\tau} = (S \setminus \bar{\tau} ) \sqcup \{v_S\} \notin K \cup \{\sigma \sqcup \{v_S\} \colon \sigma \in \hat{L}\} = \cA_{(K,L)},
    $$ 
    which confirms that~$\tau \in \hat{\cA}_{(K,L)}$.
    Therefore, $\hat{\cA}_{(K,L)} \supset \cA_{(L,K)}$.
\end{proof}

Let $K$ be a simplicial complex $K$ on $[m]= \{1,\ldots,m\}$, which is not the power set~$2^{[m]}$ of~$[m]$.
For a subset~$I\subseteq [m]$, we consider the full subcomplex~$\Bier(K)\vert_{I\sqcup \bar{I}}$ in the case where~$I \sqcup \bar{I}$.

\begin{lemma}~\label{sec3:lemma2}
  Let $K$ be a simplicial complex on $S$ and $I$ a subset of $S$.
  Then,
  $$
  \Bier(K)\vert_{I \sqcup\bar{I}} = \Bier(\cA_{(K\vert_I, (\hat{K})\vert_{\bar{I}})})\vert_{I \sqcup \bar{I}}.
  $$
\end{lemma}
\begin{proof}
For the full subcomplex $K\vert_I$ of $K$, let us consider a simplex~$\tau$ in the Alexander dual~$\widehat{K\vert_I}$ of~$K\vert_I$.
Since $\tau \in \widehat{K\vert_I}$, it follows that $I \setminus \bar{\tau}$ is not a simplex in $K\vert_I$.
Thus, we have $I \setminus \bar{\tau} \notin K$, which confirms that $\tau \in \hat{K}$.
Since $\tau \subset \bar{I}$, we conclude that $\tau$ is a simplex in the full subcomplex $(\hat{K})\vert_{\bar{I}}$ of $\hat{K}$.
Therefore, we obtain
\begin{equation}\label{sec3:eq2}
      \widehat{K\vert_I} \subset (\hat{K})\vert_{\bar{I}}
\end{equation}
as a subcomplex.

Applying \eqref{sec3:eq2} by replacing $K$ with $\hat{K}$ and $I$ with $\bar{I}$, we obtain that the Alexander dual of~$(\hat{K})\vert_{\bar{I}}$ is a subcomplex of $K\vert_I$.
Combining Lemma~\ref{sec3:lemma1} with~\eqref{sec3:eq2}, it follows that
$$
  \hat{\cA}_{(K\vert_I,(\hat{K})\vert_{\bar{I}})} = \cA_{((\hat{K})\vert_{\bar{I}},K\vert_I)}.
$$
Furthermore, by Proposition~\ref{sec3:prop1}, we have
$$
\cA_{(K\vert_I,(\hat{K})\vert_{\bar{I}})}\vert_I = K\vert_I \text{ and } \hat{\cA}_{(K\vert_I, (\hat{K})\vert_{\bar{I}})}\vert_{\bar{I}} = (\hat{K})\vert_{\bar{I}}.
$$
We conclude that
    \begin{align*}
      \Bier(K)\vert_{I\sqcup\bar{I}} & = \{\sigma \sqcup \tau \colon \sigma \in K, \tau \in \hat{K}, \bar{\sigma} \cap \tau = \emptyset, \sigma \subset I, \tau \subset \bar{I}\} \\
      &= \{\sigma \sqcup \tau \colon \sigma \in K\vert_I, \tau \in (\hat{K})\vert_{\bar{I}}, \bar{\sigma} \cap \tau = \emptyset\}\\
      &= K\vert_I \ast_{\Delta} (\hat{K})\vert_{\bar{I}}\\
      &= \cA_{(K\vert_I,(\hat{K})\vert_{\bar{I}})}\vert_I \ast_{\Delta} (\hat{\cA}_{(K\vert_I,(\hat{K})\vert_{\bar{I}})})\vert_{\bar{I}} \\
      & = \{\sigma \sqcup \tau \colon \sigma \in \cA_{(K\vert_I,(\hat{K})\vert_{\bar{I}})}, \tau \in \hat{\cA}_{(K\vert_I,(\hat{K})\vert_{\bar{I}})}, \bar{\sigma} \cap \tau = \emptyset, \sigma \subset I, \tau \subset \bar{I}\}\\
      &= \Bier(\cA_{(K\vert_I,(\hat{K})\vert_{\bar{I}})})\vert_{I\sqcup\bar{I}}.
    \end{align*}
\end{proof}

We are ready to describe the topology of the full subcomplex~$\Bier(K)\vert_{I\sqcup\bar{I}}$.
By Lemma~\ref{sec3:lemma2}, $\Bier(K)\vert_{I\sqcup\bar{I}}$ is obtainable from~$\Bier(\cA_{(K\vert_I, (\hat{K})\vert_{\bar{I}})})$ by removing the two vertices $v_I$ and $v_{\bar{I}}$, where $v_I$ and $v_{\bar{I}}$ are not necessarily non-ghost vertices.

\begin{theorem}\label{sec3:thm}
    Let $K\neq 2^{[m]}$ be a simplicial complex on $[m]$.
    For a subset $I \subseteq [m]$ of cardinality~$r \geq 1$, the full subcomplex $\Bier(K)\vert_{I\sqcup \bar{I}}$ satisfies the following:
    \begin{enumerate}
      \item it is the boundary complex of the $r$-dimensional cross-polytope, if both $I \in K$ and $\bar{I} \in \hat{K}$,
      \item it is homotopy equivalent to an $(r-2)$-sphere, if neither $I \in K$ nor $\bar{I} \in \hat{K}$, and 
      \item it is contractible otherwise,
    \end{enumerate}
    where the $(-1)$-sphere is the empty set.
\end{theorem}

\begin{proof}
To begin with, we investigate the case where $I \in K$ and $\bar{I} \in \hat{K}$.
This implies that both $K_I$ and $\hat{K}_{\bar{I}}$ are the power sets~$2^{I}$ and~$2^{\bar{I}}$, respectively.
For $I=\{i_1,\ldots,i_r\}$, we have
\begin{align*}
  \Bier(K)\vert_{I\sqcup \bar{I}} &= \{\sigma \sqcup \tau \colon \sigma \subset I, \tau \subset \bar{I}, \bar{\sigma} \cap \tau = \emptyset\} \\
  & = \{\eta \subset I \sqcup \bar{I} \colon \{i_\ell,\bar{i}_\ell\} \not\subset \eta \text{ for all } 1 \leq \ell \leq r\}\\
   & = \bigl\{\{i_1\}, \{\bar{i}_1\}\bigl\}\ast \cdots \ast \bigl\{\{i_r\}, \{\bar{i}_r\} \bigl\}.
\end{align*}
We conclude that $\Bier(K)\vert_{I\sqcup \bar{I}}$ is the boundary complex of the $r$-dimensional cross-polytope.

Subsequently, we turn our consideration to the case where either~$I\in K$ or~$\bar{I} \in \hat{K}$.
Without loss of generality, we assume that $I \in K$ and $\bar{I} \notin \hat{K}$.
Then, $v_I$ is a non-ghost vertex of~$\cA_{(K\vert_I, (\hat{K})\vert_{\bar{I}})}$, while $v_{\bar{I}}$ is a ghost vertex of~$\cA_{((\hat{K})\vert_{\bar{I}}, K\vert_I)}$.
Since the full subcomplex~$\Bier(K)\vert_{I\sqcup \bar{I}}$ is obtained from $\Bier(\cA_{(K\vert_I,(\hat{K})\vert_{\bar{I}})})$ by removing exactly one non-ghost vertex, we obtain that $\Bier(K)\vert_{I\sqcup \bar{I}}$ is contractible.

Finally, consider the case where neither $I \in K$ nor $\bar{I} \in \hat{K}$.
In this case, the full subcomplex~$\Bier(K)\vert_{I\sqcup \bar{I}}$ is obtained from $\Bier(\cA_{(K\vert_I,(\hat{K})\vert_{\bar{I}})})$ by removing the two non-ghost vertices~$v_I$ and $v_{\bar{I}}$.
Notably, there is no $1$-simplex $\{v_I,v_{\bar{I}}\}$ in $\Bier(\cA_{(K\vert_I,(\hat{K})\vert_{\bar{I}})})$.
Therefore, from~Exercise~1.31~in~\cite{Rotman_book1988}, the full subcomplex~$\Bier(K)\vert_{I\sqcup \bar{I}}$ is homotopy equivalent to an $(r-2)$-sphere.
\end{proof}

\begin{example}~\label{ex:3agt}
  Let $K$ be the simplicial complex on $\{1,2,3,4\}$ such that
  $$
  K=\{\emptyset, \{1\}, \{2\}, \{3\}, \{4\}, \{1,2\}, \{1,3\}, \{2,3\}, \{1,2,3\}\}.
  $$
  Then, its Alexander dual $\hat{K}$ is given by
  $$
  \hat{K} = \{\emptyset, \{\bar{1}\}, \{\bar{2}\}, \{\bar{3}\}, \{\bar{1},\bar{2}\}, \{\bar{1},\bar{3}\}, \{\bar{2},\bar{3}\}\},
  $$
  and the Bier sphere $\Bier(K)$ of $K$ is shown in Figure~\ref{sec2:figure1}.
  By Theorem~\ref{sec3:thm}, one can check that topological type of the full subcomplex~$\Bier(K)_{I \sqcup \bar{I}}$ for each~$I \subset \{1,2,3,4\}$.
  Refer to~Figure~\ref{sec2:figure1-2}.
    
    \begin{figure}
  \begin{tikzpicture}
        \node[left] at (-2.7,1) {$I = \{1,2\}$};
        \foreach \x/\y/\num in {-3/0/1, -2/-2/2, -4/-2/3} {
            \fill (\x, \y) circle (1.5pt);
        }
        \fill (-1,0) circle (1.5pt);
        \node at (-3,0.3) {$1$};
        \node at (-4,-2.3) {$2$};
        \node at (-2,-2.3) {$\bar{1}$};
        \node at (-1,0.3) {$\bar{2}$};
        
        \draw (-3,0) -- (-4,-2); 
        \draw (-2, -2) -- (-4,-2);
        \draw (-2,-2) -- (-1,0);
        \draw (-1,0) -- (-3,0); 
        \node[left] at (2.3,1) {$I = \{1,3,4\}$};
        \foreach \x/\y/\num in {2/0/1, 2.66/-2/3, 1.33/-2/4, 3.33/0/2, 4/-1/1} {
            \fill (\x, \y) circle (1.5pt);
        }
        \draw (0.66,-1) circle (1.5pt);
        \node at (2,0.3) {$1$};
        \node at (1.33,-2.3) {$3$};
        \node at (2.66,-2.3) {$\bar{1}$};
        \node at (3.33,0.3) {$\bar{3}$};
        \node at (4, -0.7) {$4$};
        \node at (0.66, -0.7) {$\bar{4}$};

        \draw (2,0) -- (1.33,-2);
        \draw (2,0) -- (3.33,0);
        \draw (2.66,-2) -- (1.33,-2);
        \draw (2.66,-2) -- (4,-1);
        \draw (4,-1) -- (3.33,0);

        \node[left] at (7.3,1) {$I = \{1,2,3\}$};
        \filldraw[gray!50] (7.5, 0.5) -- (6.25, -1) -- (8, -1) -- cycle;
    \filldraw[gray!50] (7.5, 0.5) -- (8, -1) -- (8.75, -0.5) -- cycle;
    \filldraw[gray!50] (7.5, 0.5) -- (8.75, -0.5) -- (7, -0.5) -- cycle;
    \filldraw[gray!50] (7.5, -2) -- (7, -0.5) -- (6.25, -1) -- cycle;
    \filldraw[gray!50] (7, -0.5) -- (8, -1) -- (7.5, -2) -- cycle;
    \filldraw[gray!20] (7.5, -2) -- (8,-1) -- (8.75, -0.5) -- cycle;
        \fill (7, -0.5) circle (1.5pt);
        \fill (6.25, -1) circle (1.5pt);
        \fill (8, -1) circle (1.5pt);
        \fill (8.75, -0.5) circle (1.5pt);
        \fill (7.5, 0.5) circle (1.5pt);
        \fill (7.5, -2) circle (1.5pt);

        \node at (7.5, 0.8) {1};
        \node at (7.2,-0.3) {2};
        \node at (6,-1) {3};
        \node at (7.5,-2.3) {$\bar{1}$};
        \node at (8.1, -0.7) {$\bar{2}$};
        \node at (8.9,-0.3) {$\bar{3}$};

        \draw[dotted] (7.5, 0.5) -- (7, -0.5);
        \draw (7.5, 0.5) -- (6.25, -1);
        \draw (7.5, 0.5) -- (8,-1);
        \draw (7.5, 0.5) -- (8.75, -0.5);
        
        \draw[dotted] (7, -0.5) -- (6.25, -1);
        \draw (6.25, -1) -- (8, -1);
        \draw (8, -1) -- (8.75, -0.5);
        \draw[dotted] (8.75, -0.5) -- (7, -0.5);
        
        \draw[dotted] (7.5, -2) -- (7, -0.5);
        \draw (7.5, -2) -- (6.25, -1);
        \draw (7.5, -2) -- (8,-1);
        \draw (7.5, -2) -- (8.75, -0.5);
    \end{tikzpicture}
    \caption{Full subcomplexes $\Bier(K)_{I \sqcup \bar{I}}$}
    \label{sec2:figure1-2}
    \end{figure}
\end{example}

\section{$\Bier(K)\vert_{I\sqcup\bar{J}}$ when $I\not\subset J$ and $I\not\supset J$}~\label{sec3}
Throughout this section, we assume that $I$ and $J$ are subsets of~$[m] = \{1,\ldots,m\}$.
We discuss the topological types of the full subcomplex~$\Bier(K)\vert_{I \sqcup \bar{J}}$ of Bier spheres~$\Bier(K)$ in the case where~$I\not\subset J$ and $I\not\supset J$.

To begin, we focus on the case where~$I$ and~$J$ are disjoint.
In this case, the full subcomplex~$\Bier(K)\vert_{I \sqcup \bar{J}}$ of the Bier sphere~$\Bier(K)$ defined as the deleted join~$K\vert_{I} \ast_{\Delta} \hat{K}\vert_{\bar{J}}$ coincides with the ordinary simplicial join~$K\vert_{I} \ast \hat{K}\vert_{\bar{J}}$, that is,
\begin{equation}\label{eq:Bierjoin}
  \Bier(K)\vert_{I \sqcup \bar{J}} = K\vert_{I} \ast \hat{K}\vert_{\bar{J}}.
\end{equation}

The following lemma was originally proved in~\cite[Lemma~$3.3$]{Heudtlass-Katthan2012}, and for the reader's convenience we provide a more intuitive proof here.

\begin{lemma}~\label{lem:simplex}
  Let $K$ be a simplicial complex on $[m]=\{1,2,\ldots,m\}$, and $I,J$ nonempty subsets of $[m]$ with~$I\cap J=\emptyset$.
  Then, either $K\vert_I$ or $\hat{K}\vert_{\bar{J}}$ is a simplex.
  Moreover, the full subcomplex~$\Bier(K)\vert_{I\sqcup\bar{J}}$ is contractible. 
\end{lemma}
\begin{proof}
  If $I$ is a simplex in $K$, then $K\vert_I$ is $(\left\vert I \right\vert - 1)$-simplex. 
  Thus, we may assume that there is a subset $\sigma \subset I$ which is not a simplex in $K\vert_I$.
  By the definition of the Alexander dual, $[\bar{m}]\setminus\bar{\sigma}$ is a simplex in $\hat{K}$.
  Since $I$ and $J$ are disjoint, we obtain
  $$
  \bar{J} \subset [\bar{m}]\setminus \bar{I} \subset [\bar{m}]\setminus \bar{\sigma} \in \hat{K}.
  $$
  Since $\hat{K}$ is a simplicial complex, we have $\bar{J} \in \hat{K}$, and then $\hat{K}\vert_{\bar{J}}$ is a $(\left\vert \bar{J}\right\vert - 1)$-simplex.
  Therefore, either $K\vert_I$ or $\hat{K}\vert_{\bar{J}}$ is always a simplex.
  
  It is well-known~\cite{Whitehead1978} that the simplicial join of any space with a contractible space is contractible.
  Hence, we conclude that~$\Bier(K)\vert_{I\sqcup\bar{J}} = K\vert_{I} \ast \hat{K}\vert_{\bar{J}}$ is contractible.
\end{proof}

For any simplicial complex~$\Gamma$ and subsets~$I_0, J_0$ of its vertex set, we have~$\Gamma\vert_{I_0} \cap \Gamma\vert_{J_0} = \Gamma\vert_{I_0 \cap J_0}$ as sets.
Then it follows immediately that
\begin{equation}\label{eq:full-inter}
  \Bier(K)\vert_{I_1 \sqcup \bar{J}_1}\cap \Bier(K)\vert_{I_2 \sqcup \bar{J}_2} = \Bier(K)\vert_{(I_1 \cap I_2) \sqcup (\bar{J}_1 \cap \bar{J}_2)},
\end{equation}
where $I_1, I_2, J_1, J_2$ are subsets of~$[m]$.
Moreover, for any $I_1\subset I_2$ and $J_1 \subset J_2$, we have
\begin{equation}\label{eq:inc1}
  \Bier(K)\vert_{I_1 \sqcup \bar{J}_1} \subset \Bier(K)\vert_{I_2 \sqcup \bar{J}_2}
\end{equation}
as a subcomplex.

For each~$\sigma \in K\vert_{I \cap J}$, define
\begin{equation}\label{eq:cover}
  U_{\sigma} := K\vert_{(I\setminus (I\cap J)) \cup \sigma} \ast \hat{K}\vert_{\bar{J}\setminus \bar{\sigma}}.
\end{equation}
Note that~$(I\setminus (I\cap J)) \cup\sigma$ and $J\setminus\sigma$ are disjoint for each $\sigma\subseteq I\cap J$.
From~\eqref{eq:Bierjoin}, $U_\sigma$ can be also regarded as a full subcomplex
\begin{equation}\label{eq:cov_Bier}
    U_\sigma =\Bier(K)\vert_{((I\setminus (I\cap J)) \cup \sigma) \sqcup (\bar{J}\setminus \bar{\sigma})}
\end{equation}
of~$\Bier(K)$.

\begin{lemma}\label{lem:cover}
  Let $K$ be a simplicial complex on $[m]$ and $I, J$ subsets of $[m]$.
  Then, 
  $$
  \Bier(K)\vert_{I \sqcup \bar{J}}=
  \bigcup_{\sigma \in K\vert_{I\cap J}} U_\sigma.
  $$
\end{lemma}
\begin{proof}
  To begin, by~\eqref{eq:inc1} and~\eqref{eq:cov_Bier}, we have
  $$
  U_\sigma\subset \Bier(K)\vert_{I \sqcup \bar{J}} \mbox{ as a subcomplex,}
  $$
  for each $\sigma\in K\vert_{I\cap J}.$
  Next, we assume that~$\sigma \sqcup \bar{\tau} \in \Bier(K)\vert_{I \sqcup \bar{J}}$, where~$\sigma \in K\vert_I$, $\bar{\tau} \in \hat{K}\vert_{\bar{J}}$, and $\sigma \cap \tau = \emptyset$.
  Define $\sigma' \in K\vert_{I\cap J}$ by
  $$
  \sigma' = \sigma \cap (I\cap J).
  $$
  Then $\sigma$ can be separated into two parts as
  $$
  \sigma = \big(\sigma\cap (I\setminus (I\cap J)) \big) \sqcup \sigma', 
  $$
  and it follows that~$\sigma \in K\vert_{(I \setminus (I\cap J)) \cup \sigma'}$.
  Since~$\sigma$ and~$\tau$ are disjoint, $\sigma' \cap \tau = \emptyset$, and hence~$\bar{\tau} \in \hat{K}\vert_{\bar{J}\setminus \bar{\sigma}'}$.
  Therefore, we conclude that~$\sigma \sqcup \bar{\tau} \in U_{\sigma'}$,
  as desired.
\end{proof}

\begin{example}
    Consider the simplicial complex~$K$ on~$\{1,2,3,4\}$ in Example~\ref{ex:3agt}.
    Let~$I = J = \{1,2\}$.
    Then~$K\vert_{I \cap J} = \{\emptyset,\{1\},\{2\},\{1,2\}\}$.
    We obtain that
    $$
    U_{\emptyset} = \hat{K}_{\{\bar{1},\bar{2}\}} , U_{\{1\}} = K\vert_{\{1\}}\ast_{\Delta}  \hat{K}\vert_{\{\bar{2}\}}, U_{\{2\}} = K\vert_{\{2\}} \ast_\Delta \hat{K}\vert_{\{\bar{1}\}}, \text{ and }  U_{\{1,2\}} = K\vert_{\{1,2\}}.
    $$
    Hence, the full subcomplex~$\Bier(K)\vert_{\{1,2,\bar{1},\bar{2}\}}$ is covered by~$\{U_\sigma\}_{\sigma \in K\vert_{I\cap J}}$ according to Lemma~\ref{lem:cover}.
    See Figure~\ref{sec3:figure1} for details.
\begin{figure}
  \begin{tikzpicture}
        \node[left] at (-2.7,1) {\large{$\Bier(K)\vert_{I \sqcup \bar{I}}$}};
        \foreach \x/\y/\num in {-3/0/1, -2/-2/2, -4/-2/3} {
            \fill (\x, \y) circle (1.5pt);
        }
        \fill (-1,0) circle (1.5pt);
        \node at (-3,0.3) {$1$};
        \node at (-4,-2.3) {$2$};
        \node at (-2,-2.3) {$\bar{1}$};
        \node at (-1,0.3) {$\bar{2}$};
        
        \draw (-3,0) -- (-4,-2); 
        \draw (-2, -2) -- (-4,-2); 
        \draw (-2,-2) -- (-1,0);
        \draw (-1,0) -- (-3,0); 

        \draw[->, thick] (-1.1,-1) -- (-0.1,-1);

        \draw[blue] (1,0) -- (0.33,-2);
        \node at (0.3, -1) {\rotatebox{77}{$K\vert_{\{1,2\}}$}};
        \draw[red] (2,0) -- (3.33,0);
        \node at (2.66, 1) {$K\vert_{\{1\}}\ast_\Delta \hat{K}\vert_{\{\bar{2}\}}$};
        \draw[green] (2.66,-2) -- (1.33,-2);
        \node at (1.99, -3) {$K\vert_{\{2\}}\ast_\Delta \hat{K}\vert_{\{\bar{1}\}}$};
        \draw[yellow] (4.33,0) -- (3.66,-2);
        \node at (4.33, -1) {\rotatebox{-103}{$\hat{K}\vert_{\{\bar{1},\bar{2}\}}$}};
        \foreach \x/\y/\num in {2/0/1, 2.66/-2/3, 1.33/-2/4, 3.33/0/2} {
            \fill (\x, \y) circle (1.5pt);
        }
        \node at (2,0.3) {$1$};
        \node at (1.33,-2.3) {$2$};
        \node at (2.66,-2.3) {$\bar{1}$};
        \node at (3.33,0.3) {$\bar{2}$};
        \node at (1,0.3) {$1$};
        \node at (0.33,-2.3) {$2$};
        \node at (3.66,-2.3) {$\bar{1}$};
        \node at (4.33,0.3) {$\bar{2}$};
        
        \fill (1,0) circle (1.5pt);
        \fill (0.33,-2) circle (1.5pt);
        \fill (4.33,0) circle (1.5pt);
        \fill (3.66,-2) circle (1.5pt);

    \end{tikzpicture}
  \caption{A cover of a full subcomplex}
  \label{sec3:figure1}
\end{figure}  
\end{example}

  For a cover~$\cU = \{U_{i}\}_{i\in \cI}$ of a simplicial complex~$\Gamma$, the \emph{nerve} $\cN(\cU)$ of $\cU$ is the simplicial complex on $\cI$, defined by
  $$
  \cN(\cU)=\{\sigma \subset \mathcal{I} : \bigcap\limits_{i\in\sigma} U_i \neq \emptyset\}.
  $$
 
 The following lemma is known as the \emph{nerve lemma}.
  \begin{lemma}\label{lem:nerve}\cite{leray1950}
  Let $\Gamma$ be a simplicial complex with a cover~$\cU = \{U_{i}\}_{i\in \cI}$.
  If all finite intersections of the covering sets~$U_i \in \cU$ are contractible, then the simplicial complex~$\Gamma$ is homotopy equivalent to the nerve~$\cN(\cU)$ of~$\cU$.
  \end{lemma}
  
Let us consider
$$
\cU := \{U_{\sigma}\}_{\sigma \in K_{I \cap J}},
$$
where~$U_{\sigma}$ is defined in~\eqref{eq:cover}.
By Lemma~\ref{lem:cover}, $\cU$ is a cover of a full subcomplex~$\Bier(K)\vert_{I\sqcup \bar{J}}$.
Using this cover, we prove Theorem~\ref{thm:contrac}.
  
\begin{theorem}\label{thm:contrac}
  Let $K$ be a simplicial complex on $[m] =\{1,\ldots,m\}$, and $I,J$ nonempty subsets of $[m]$.
  If $I\not\subset J$ and $I\not\supset J$, then the full subcomplex $\Bier(K)\vert_{I\sqcup \bar{J}}$ is contractible.
\end{theorem}
\begin{proof}
Since $I\not\subset J$ and $I\not\supset J$, both $I\setminus (I\cap J)$ and $J\setminus (I\cap J)$ are nonempty and disjoint.
Then we have
\begin{align*}
  \bigcap_{\sigma\in K\vert_{I\cap J}} U_{\sigma} 
  &= \bigcap_{\sigma\in K\vert_{I\cap J}} \Bier(K)\vert_{((I\setminus (I\cap J)) \cup \sigma) \sqcup (\bar{J}\setminus \bar{\sigma})}
    && \text{(by~\eqref{eq:cov_Bier})} \\
  &= \Bier(K)\vert_{(I\setminus(I\cap J)) \sqcup (\bar{J}\setminus(\bar{I}\cap\bar{J}))}
    && \text{(by repeated application of~\eqref{eq:full-inter})} \\
  &= K\vert_{I\setminus(I\cap J)} \ast \hat{K}\vert_{\bar{J}\setminus(\bar{I}\cap\bar{J})}
    && \text{(by~\eqref{eq:Bierjoin})}.
\end{align*}
  From Lemma~\ref{lem:simplex}, we observe that either $K\vert_{I\setminus(I\cap J)}$ or $\hat{K}\vert_{\bar{J}\setminus(\bar{I}\cap \bar{J})}$ is a simplex.
  Hence,
  $$
  \bigcap_{\sigma\in K\vert_{I\cap J}} U_{\sigma} \neq \emptyset,
  $$
  which is implies that the nerve~$\cN(\cU)$ is a simplex.
  
  Let~$\sigma_1,\ldots,\sigma_\ell$ be simplices in~$K\vert_{I\cap J}$.
  Since both
  $$
  \mathbf{I} := (I\setminus (I\cap J)) \cup (\sigma_1 \cap \cdots \cap \sigma_\ell)\quad \text{and} \quad \mathbf{J} := J\setminus (\sigma_1 \cup \cdots \cup \sigma_\ell)
  $$
  are nonempty and disjoint, we obtain that
  $$
  \bigcap_{1 \leq i \leq \ell} U_{\sigma_i} = \Bier(K)_{\mathbf{I} \sqcup \bar{\mathbf{J}}}
  $$
  is contractible, by Lemma~\ref{lem:simplex}.
  Therefore, by Lemma~\ref{lem:nerve}, $\Bier(K)\vert_{I\sqcup\bar{J}}$ is homotopy equivalent to~$\cN(\cU)$.
\end{proof}

\begin{example}
  Recall the simplicial complex~$K$ from Example~\ref{ex:3agt}.
  Consider
  $$
  I_1 = \{2,4\}, J_1=\{1,3\} , I_2 = \{2,3,4\}, \text{ and } J_2 = \{1,2,3\}.
  $$
  By Theorem~\ref{thm:contrac}, the full subcomplexes $\Bier(K)\vert_{I_1 \sqcup \bar{J}_1}$ and $\Bier(K)\vert_{I_2 \sqcup \bar{J}_2}$ are both contractible.
  See Figure~\ref{sec3:figure2}.
\begin{figure}
  \begin{tikzpicture}
        \node[left] at (6.6,1) {\large{$\Bier(K)$}};
        \filldraw[gray!20] (7.5, 0.5) -- (6.25, -1) -- (8, -1) -- cycle;
    \filldraw[gray!20] (7.5, 0.5) -- (8, -1) -- (8.75, -0.5) -- cycle;
    \filldraw[gray!20] (7.5, -2) -- (6.25, -1) -- (8, -1) -- cycle;
    \filldraw[gray!20] (7.5, -2) -- (8, -1) -- (8.5, -1.5) -- cycle;
    \filldraw[gray!20] (8.75, -0.5) -- (8, -1) -- (8.5, -1.5) -- cycle;
    \filldraw[gray!20] (7.5, 0.5) -- (6.25, -1) -- (7, -0.5) -- cycle;
    \filldraw[gray!20] (7.5, 0.5) -- (8.75, -0.5) -- (7, -0.5) -- cycle;
    \filldraw[gray!20] (7.5, -2) -- (7, -0.5) -- (6.25, -1) -- cycle;
    \filldraw[gray!20] (7.5, -2) -- (7, -0.5) -- (8.75, -0.5) -- cycle;
        \fill (7, -0.5) circle (1.5pt);
        \fill (6.25, -1) circle (1.5pt);
        \fill (8, -1) circle (1.5pt);
        \fill (8.75, -0.5) circle (1.5pt);
        \fill (7.5, 0.5) circle (1.5pt);
        \fill (7.5, -2) circle (1.5pt);
        \fill (8.5, -1.5) circle (1.5pt);
        \draw (6.5, 0) circle (1.5pt);

        \node at (7.5, 0.8) {1};
        \node at (7.2,-0.3) {2};
        \node at (6,-1) {3};
        \node at (8.7, -1.6) {4};
        \node at (7.5,-2.3) {$\bar{1}$};
        \node at (8.1, -0.7) {$\bar{2}$};
        \node at (8.9,-0.3) {$\bar{3}$};
        \node at (6.5, 0.3) {$\bar{4}$};

        \draw[dotted] (7.5, 0.5) -- (7, -0.5);
        \draw (7.5, 0.5) -- (6.25, -1);
        \draw (7.5, 0.5) -- (8,-1);
        \draw (7.5, 0.5) -- (8.75, -0.5);
        
        \draw[dotted] (7, -0.5) -- (6.25, -1);
        \draw (6.25, -1) -- (8, -1);
        \draw (8, -1) -- (8.75, -0.5);
        \draw[dotted] (8.75, -0.5) -- (7, -0.5);
        
        \draw[dotted] (7.5, -2) -- (7, -0.5);
        \draw (7.5, -2) -- (6.25, -1);
        \draw (7.5, -2) -- (8,-1);
        \draw[dotted] (7.5, -2) -- (8.75, -0.5);
        
        \draw (8.5, -1.5) -- (8.75, -0.5);
        \draw (8.5, -1.5) -- (8, -1);
        \draw (8.5, -1.5) -- (7.5, -2);

        \fill[fill=gray!50] (12,0)--(11,-2)--(13,-2);
        \fill[fill=gray!50] (12,0)--(13,-2)--(14,0);
        \node[left] at (12,1) {\large{$\Bier(K)\vert_{I_1 \sqcup \bar{J}_1}$}};
        \foreach \x/\y/\num in {12/0/1, 13/-2/3, 11/-2/2} {
            \fill (\x, \y) circle (1.5pt);
        }
        \fill (14,0) circle (1.5pt);
        \node at (12,0.3) {$\bar{3}$};
        \node at (11,-2.3) {2};
        \node at (13,-2.3) {$\bar{1}$};
        \node at (14,0.3) {4};
        
        \draw (12,0) -- (11,-2);
        \draw (12,0) -- (13,-2);
        \draw (13,-2) -- (11,-2);
        
        \draw (12,0)--(14,0);
        \draw (13,-2)--(14,0);

        \node[left] at (17.6,1) {\large{$\Bier(K)\vert_{I_2 \sqcup \bar{J}_2}$}};
        \filldraw[gray!20] (17, -0.5)--(16.25,-1)--(17.5,-2)--cycle;
        \filldraw[gray!20] (17,-0.5)--(18.75,-0.5)--(17.5,-2)--cycle;
    \filldraw[gray!50] (17.5, -2) -- (16.25, -1) -- (18, -1) -- cycle;
    \filldraw[gray!50] (17.5, -2) -- (18, -1) -- (18.5, -1.5) -- cycle;
    \filldraw[gray!50] (18.75, -0.5) -- (18, -1) -- (18.5, -1.5) -- cycle;
        \fill (17, -0.5) circle (1.5pt);
        \fill (16.25, -1) circle (1.5pt);
        \fill (18, -1) circle (1.5pt);
        \fill (18.75, -0.5) circle (1.5pt);
        \fill (17.5, -2) circle (1.5pt);
        \fill (18.5, -1.5) circle (1.5pt);

        \node at (17.2,-0.3) {2};
        \node at (16,-1) {3};
        \node at (18.7, -1.6) {4};
        \node at (17.5,-2.3) {$\bar{1}$};
        \node at (18.1, -0.7) {$\bar{2}$};
        \node at (18.9,-0.3) {$\bar{3}$};

        \draw (17, -0.5) -- (16.25, -1);
        \draw (16.25, -1) -- (18, -1);
        \draw (18, -1) -- (18.75, -0.5);
        \draw (18.75, -0.5) -- (17, -0.5);
        
        \draw[dotted] (17.5, -2) -- (17.165, -1);
        \draw (17.165, -1) -- (17,-0.5);
        \draw (17.5, -2) -- (16.25, -1);
        \draw (17.5, -2) -- (18,-1);
        \draw[dotted] (17.5, -2) -- (18.75, -0.5);
        
        \draw (18.5, -1.5) -- (18.75, -0.5);
        \draw (18.5, -1.5) -- (18, -1);
        \draw (18.5, -1.5) -- (17.5, -2);
    \end{tikzpicture}
    \caption{Full subcomplexes $\Bier(K)_{I \sqcup \bar{J}}$}
    \label{sec3:figure2}
    \end{figure}
\end{example}

\section{$\Bier(K)\vert_{I\sqcup\bar{J}}$ when $I\subsetneq J$ or $I\supsetneq J$}\label{sec:inclu}

By the commutativity of deleted join, $\Bier(K)$ coincides with $\Bier(\hat{K})$, that is,
$$
\Bier(K)\vert_{I\sqcup\bar{J}} = \Bier(\hat{K})\vert_{\bar{J}\sqcup I}.
$$
Hence, in this section, we may restrict to the case~$J \subsetneq I$.

Let $\Gamma$ be a simplicial complex on $S$ and $\sigma$ a simplex in~$\Gamma$.
The \emph{link} $\Lk(\sigma,\Gamma)$ of $\sigma$ in $\Gamma$ is defined as
$$
\Lk(\sigma, \Gamma) = \{\tau \in \Gamma : \sigma \cap \tau =\emptyset, \sigma \cup \tau \in \Gamma\}.
$$

\begin{lemma}~\label{lem:lk-lk}
  Let $\Gamma$ be a simplicial complex and $\sigma$ a simplex in $\Gamma$.
  If $\tau$ is a simplex in~$\Lk(\sigma,\Gamma)$, then
  $$
  \Lk(\tau, \Lk(\sigma,\Gamma)) = \Lk(\sigma \sqcup \tau, \Gamma).
  $$
\end{lemma}
\begin{proof}
    Suppose that $\eta \in \Lk(\tau, \Lk(\sigma,\Gamma))$.
    It follows that~$\eta \sqcup \tau \in \Lk(\sigma,\Gamma)$, and hence
    $$
    \eta \sqcup (\sigma\sqcup\tau) = (\eta \sqcup \tau)  \sqcup \sigma \in \Gamma.
    $$
    Therefore, $\eta \in \Lk(\sigma\sqcup\tau, \Gamma)$.
    
    Next, suppose that $\eta \in \Lk(\sigma\sqcup\tau, \Gamma)$.
    Then we obtain that
    $$
    (\eta \sqcup \tau) \sqcup \sigma = \eta \sqcup (\sigma\sqcup\tau) \in \Gamma,
    $$
    which implies that~$\eta \sqcup \tau \in \Lk(\sigma,\Gamma)$.
    We conclude that~$\eta \in \Lk(\tau,\Lk(\sigma,\Gamma))$.
\end{proof}

Let $\Gamma$ be a simplicial complex.
The \emph{suspension} $\Sigma\Gamma$ of $\Gamma$ is defined as
$$
\Sigma\Gamma = (v_1 \ast \Gamma) \cup (v_2 \ast \Gamma),
$$
where $v_1$ and $v_2$, called the \emph{apices} of the suspension~$\Sigma\Gamma$, are additional non-adjacent vertices not contained in~$\Gamma$.
We denote the $k$-fold suspension by~$\Sigma^k$. 
In particular, $\Sigma^0 \Gamma = \Gamma$.

For a simplex~$\tau$ in~$\Gamma$, we set
$$
  \Gamma \setminus \tau := \{\eta \in \Gamma : \eta \cap \tau = \emptyset\}.
$$
Note that, for an additional vertex~$v$ not in~$\Gamma$,
\begin{equation}\label{eq:stardelete}
  (v \ast \Gamma)\setminus \tau = v \ast (\Gamma\setminus \tau).
\end{equation}

\begin{proposition}\label{prop:Sigmalink}
Let $\Gamma$ be a finite simplicial complex and $\sigma$ a simplex in~$\Gamma$. Then,
\begin{enumerate}
  \item\label{prop:Sigmalink1}
      $\Lk(\sigma,\Sigma\Gamma) \cong \Sigma \Lk(\sigma,\Gamma)$ as simplicial complexes, and
  \item\label{prop:Sigmalink2} for each $\tau \in \Lk(\sigma,\Gamma)$, $(\Sigma\Lk(\sigma,\Gamma))\setminus \tau \cong \Sigma\Lk(\sigma, \Gamma\setminus \tau)$ as simplicial complexes.
\end{enumerate}
\end{proposition}
\begin{proof}
Whenever a suspension appears in this proof, its two apices are identified with~$v_1$ and~$v_2$.
We compute
\begin{align*}
  \Lk(\sigma,\Sigma\Gamma) & =\{\tau\in\Sigma\Gamma : \sigma\sqcup\tau \in \Sigma\Gamma\} \\
   & = \{\tau\in v_1 \ast \Gamma : \sigma\sqcup\tau \in v_1 \ast \Gamma\} \cup \{\tau\in v_2 \ast \Gamma : \sigma\sqcup\tau \in v_2 \ast \Gamma\} .
\end{align*}
Since $\sigma \in \Gamma$, it contains neither~$v_1$ nor~$v_2$. Thus, for~$i =1,2$,
$$
\{\tau\in v_i \ast \Gamma : \sigma\sqcup\tau \in v_i \ast \Gamma\} = (v_i \ast \{\eta \in \Gamma : \sigma \sqcup \eta \in \Gamma\}) = v_i \ast \Lk(\sigma,\Gamma).
$$
We conclude that
\begin{align*}
  \Lk(\sigma,\Sigma\Gamma)
  &= v_1 \ast \Lk(\sigma,\Gamma) \cup v_2 \ast \Lk(\sigma,\Gamma) \\
   & = \Sigma\Lk(\sigma,\Gamma),
\end{align*}
which confirms that~\eqref{prop:Sigmalink1}.

Now, let~$\tau \in \Lk(\sigma,\Gamma)$.
We have
\begin{align*}
  (\Sigma\Lk(\sigma,\Gamma))\setminus\tau  & = \Big(\big(v_1\ast \Lk(\sigma,\Gamma)\big)\cup \big(v_2\ast\Lk(\sigma,\Gamma)\big)\Big)\setminus\tau\\
   & = \Big(\big(v_1\ast\Lk(\sigma,\Gamma)\big)\setminus\tau\Big)\cup \Big(\big(v_2\ast\Lk(\sigma,\Gamma)\big)\setminus\tau\Big) \\
   & = \Big(v_1\ast\big(\Lk(\sigma,\Gamma)\setminus\tau\big)\Big) \cup \Big(v_2\ast\big(\Lk(\sigma,\Gamma)\setminus\tau\big)\Big) \quad \text{(by~\eqref{eq:stardelete})}\\
   & = \Sigma(\Lk(\sigma,\Gamma)\setminus\tau).
\end{align*}
On the other hand, we compute
\begin{align*}
  \Lk(\sigma,\Gamma)\setminus\tau & = \{\eta \in \Lk(\sigma, \Gamma) : \eta\cap\tau = \emptyset\} \\
   & =\{\eta \in \Gamma : \eta\sqcup\sigma \in \Gamma, \eta\cap\tau = \emptyset\} \\
   & =\{\eta \in \Gamma\setminus\tau : \eta\sqcup\sigma\in \Gamma\setminus\tau\}\\
   & = \Lk(\sigma,\Gamma\setminus\tau).
\end{align*}
Combining the two computations above completes the proof of~\eqref{prop:Sigmalink2}.
\end{proof}

We introduce a key lemma for the proof of Theorem~\ref{thm:4_main}.
\begin{lemma}\label{lemma:delete}
    Let $u$ be a vertex of~$\Gamma$.
    If~$\Gamma \setminus u$ is contractible, then
    $$
    \Gamma \simeq \Sigma\Lk(u,\Gamma).
    $$
\end{lemma}
\begin{proof}
    Note that the inclusion
$$
\Lk(u,\Gamma) \hookrightarrow \Gamma\setminus u
$$
is a null homotopy.
Then it is known~\cite{Adamaszek-Adams2022} that 
$$
  \Gamma \simeq (\Gamma\setminus u) \vee \Sigma \Lk(u,\Gamma),
$$
where~$X \vee Y$ is the wedge sum of topological spaces~$X$ and~$Y$.
This proves the lemma.
\end{proof}

\begin{lemma}\label{lemma:sigma}
  Let~$\eta$ be a simplex in~$K$ with $\eta \subset I\cap J$.
  Then,
$$
  \Lk(\eta,\Bier(K)\vert_{I\sqcup \bar{J}}) = \Lk(\eta,\Bier(K)\vert_{I\sqcup(\bar{J}\setminus\bar{\eta})}) = \Lk(\eta, K\vert_I)\ast_{\Delta} \hat{K}\vert_{\bar{J}\setminus\bar{\eta}}.
    $$
\end{lemma}
\begin{proof}
    It is immediate that the first term contains the second one as a subcomplex.
    Now we assume that~$\sigma \sqcup \bar{\tau}$ is a simplex in~$\Lk(\eta,\Bier(K)\vert_{I\sqcup\bar{J}})$, where~$\sigma \in K$ and~$\tau \in \hat{K}$.
    Then~$\eta \sqcup (\sigma\sqcup\bar{\tau})$ is a simplex in~$ \Bier(K)\vert_{I\sqcup\bar{J}}$.
    Since
    $$
    \eta \sqcup (\sigma\sqcup\bar{\tau}) = (\eta \sqcup \sigma)\sqcup\bar{\tau} \in \Bier(K)_{I \sqcup \bar{J}},
    $$
    we deduce that~$\eta \cap \tau = \emptyset$, and hence~$\bar{\tau} \in \hat{K}\vert_{\bar{J} \setminus \bar{\eta}}$.
    Thus, $\eta \sqcup (\sigma \sqcup \bar{\tau})$ is a simplex in~$\Bier(K)_{I \sqcup (\bar{J} \setminus \bar{\eta})}$.
    In conclusion, $\sigma \sqcup \bar{\tau}$ is a simplex in~$\Lk(\eta,\Bier(K)\vert_{I\sqcup\bar{J}\setminus\bar{\eta}})$, which shows that the first equality holds.
    
    Moreover, the following computation verifies the second equality:
    \begin{align*}
      \Lk(\eta,\Bier(K)\vert_{I\sqcup(\bar{J}\setminus\bar{\eta})}) & = \{\sigma \sqcup \bar{\tau} \in \Bier(K)\vert_{I\sqcup(\bar{J}\setminus\bar{\eta})} \colon \eta \sqcup (\sigma \sqcup \bar{\tau}) \in \Bier(K)\vert_{I\sqcup(\bar{J}\setminus\bar{\eta})}\} \\
      & = \{\sigma \sqcup \bar{\tau} \in K_I \ast_{\Delta} \hat{K}_{\bar{J}\setminus \bar{\eta}} \colon (\eta \sqcup \sigma) \in K_I, \bar{\tau} \in \hat{K}_{\bar{J}\setminus \bar{\eta}}, (\eta \sqcup \sigma) \cap \tau = \emptyset\} \\
      &=\{\sigma \in K\vert_I \colon \eta \sqcup \sigma \in K\vert_I\} \ast_{\Delta} \hat{K}_{\bar{J}\setminus \bar{\eta}}.
    \end{align*}
\end{proof}

We now consider a full subcomplex~$\Bier(K)\vert_{I \sqcup \bar{J}}$ in the case where~$J \subsetneq I$.
Note that, for each simplex~$\sigma \in K\vert_I$ of~$\Bier(K)\vert_{I \sqcup \bar{J}}$,
\begin{equation}\label{eq:Bier_delte}
  \big( \Bier(K)\vert_{I \sqcup \bar{J}}\big) \setminus \sigma =  \Bier(K)\vert_{(I\setminus \sigma)\sqcup \bar{J}}.
\end{equation}

\begin{theorem}~\label{thm:4_main}
  Let $K$ be a simplicial complex on $[m] = \{1,\ldots,m\}$ and $J\subsetneq I \subset [m]$.
  If $\eta \subset J$ is a simplex in~$K$, then
  $$
  \Bier(K)\vert_{I\sqcup\bar{J}} \simeq \Sigma^{\left\vert \eta \right\vert}\Lk(\eta,\Bier(K)\vert_{I\sqcup (\bar{J}\setminus\bar{\eta})}).
  $$
  Consequently, if $J\in K$, then
  $$
  \Bier(K)\vert_{I\sqcup\bar{J}} \simeq \Sigma^{\left\vert J \right\vert}\Lk(J, K\vert_{I}).
  $$
\end{theorem}
\begin{proof}
    If~$J$ is empty, the statement follows immediately.
    Thus, we assume that~$J$ is non-empty.
    
    Let $v \in \eta$.
  We note that~$I \setminus v$ and~$J$ are both nonempty and are not comparable under inclusion.
  From Theorem~\ref{thm:contrac} with~\eqref{eq:Bier_delte},
  $$
  \big(\Bier(K)\vert_{I \sqcup \bar{J}}\big) \setminus v = \Bier(K)\vert_{(I\setminus v)\sqcup \bar{J}}
  $$
  is contractible.
  By Lemma~\ref{lemma:delete}, we observe that
  \begin{equation}\label{eq:mainproof}
   \Bier(K)\vert_{I\sqcup \bar{J}} \simeq \Sigma\Lk(v, \Bier(K)\vert_{I\sqcup \bar{J}}). 
  \end{equation}
  
  Now, suppose that~$v' \in \eta \setminus v$.
  By replacing~$I$ with~$I \setminus \{v'\}$ in~\eqref{eq:mainproof}, we obtain that
  \begin{align*}
    \Bier(K)\vert_{(I \setminus v') \sqcup \bar{J}} & \simeq \Sigma\Lk(v, \Bier(K)\vert_{(I \setminus v')\sqcup \bar{J}}) \\
     & = (\Sigma\Lk(v,\Bier(K)\vert_{I\sqcup\bar{J}}))\setminus v' \quad \text{(by Proposition~\ref{prop:Sigmalink}~\eqref{prop:Sigmalink2})},  \end{align*}
  they are contractible by Theorem~\ref{thm:contrac}.
  Applying Lemma~\ref{lemma:delete} by putting~$\Gamma = \Sigma\Lk(v,\Bier(K)\vert_{I\sqcup\bar{J}})$ and~$u = v'$, we obtain
  \begin{equation}\label{eq:mainproof2}
     \Sigma\Lk(v,\Bier(K)\vert_{I\sqcup\bar{J}}) \simeq \Sigma\Lk(v',\Sigma\Lk(v,\Bier(K)\vert_{I\sqcup\bar{J}})).
  \end{equation}
  We conclude that
  \begin{align*}
    \Bier(K)\vert_{I\sqcup \bar{J}} & \simeq \Sigma\Lk(v',\Sigma\Lk(v,\Bier(K)\vert_{I\sqcup\bar{J}})) & \mbox{ (by~\eqref{eq:mainproof} and~\eqref{eq:mainproof2})}\\
     & \simeq \Sigma^2 \Lk(v',\Lk(v,\Bier(K)\vert_{I\sqcup\bar{J}})) & \text{(by Proposition~\ref{prop:Sigmalink}~\eqref{prop:Sigmalink1})} \\
     & \simeq \Sigma^2 \Lk(\{v,v'\},\Bier(K)\vert_{I\sqcup\bar{J}}) & \mbox{ (by Lemma~\ref{lem:lk-lk})}
  \end{align*}
  By repeating this argument, we obtain
  $$
  \Bier(K)\vert_{I\sqcup\bar{J}} \simeq \Sigma^{\left\vert \eta \right\vert} \Lk(\eta,\Bier(K)\vert_{I\sqcup \bar{J}}).
  $$
  This completes the proof, using the first equality of Lemma~\ref{lemma:sigma}.
\end{proof}

\begin{corollary}~\label{cor:case3_contractible}
  Let $K$ be a simplicial complex on $[m]$ and $J\subsetneq I \subset [m]$.
  If $J\not\in K$, then the full subcomplex~$\Bier(K)\vert_{I\sqcup \bar{J}}$ is contractible.
\end{corollary}
\begin{proof}
  Assume that $\eta$ is a facet of $K\vert_{J}$.
  Since~$K\vert_J$ is a full subcomplex of~$K\vert_I$, $\eta$ is also a facet of~$K\vert_I$.
  For each~$v \in J\setminus\eta$, it follows that~$v \cup \eta \notin K\vert_I$, and hence~$v \notin \Lk(\sigma,K\vert_I)$.
  Consequently,
  \begin{equation}\label{eq:proof_in_corr}
    \Lk(\eta, K\vert_I) = \Lk(\eta, K\vert_I) \setminus (J \setminus \eta).
  \end{equation}

  Now, we compute
  \begin{align*}
    \Lk(\eta,\Bier(K)\vert_{I\sqcup(\bar{J}\setminus\bar{\eta})}) & = \Lk(\eta, K\vert_I)\ast_{\Delta} \hat{K}\vert_{\bar{J}\setminus\bar{\eta}} && \text{(by the second equality of Lemma~\ref{lemma:sigma})}\\
     & =\big(\Lk(\eta, K\vert_I) \setminus (J \setminus \eta) \big) \ast_{\Delta} \hat{K}\vert_{\bar{J}\setminus\bar{\eta}} && \text{(by \eqref{eq:proof_in_corr})} \\
     &= \Lk(\eta, K\vert_{I\setminus(J\setminus \eta)}) \ast_{\Delta}\hat{K}_{\bar{J}\setminus\bar{\eta}} && \text{(by \eqref{eq:Bier_delte})} \\
     &= \Lk(\eta, K\vert_{I\setminus(J\setminus \eta)}) \ast \hat{K}_{\bar{J}\setminus\bar{\eta}}. && \text{(by \eqref{eq:Bierjoin})}
  \end{align*}
  Since $I\setminus(J\setminus\eta)$ and $J\setminus\eta$ are disjoint, Lemma~\ref{lem:simplex} implies that either $K\vert_{I\setminus(J\setminus \eta)}$ or $\hat{K}\vert_{\bar{J}\setminus\bar{\eta}}$ is a simplex.
  If $K\vert_{I\setminus(J\setminus \eta)}$ is a simplex, then $\Lk(\eta, K\vert_{I\setminus(J\setminus \eta)})$ is also a simplex.
  Thus, either $\Lk(\eta, K\vert_{I\setminus(J\setminus \eta)})$ or $\hat{K}_{\bar{J}\setminus\bar{\eta}}$ is contractible.
  In either cases, $\Lk(\eta,\Bier(K)\vert_{I\sqcup(\bar{J}\setminus\bar{\eta})})$ is always contractible.
  
  By Theorem~\ref{thm:4_main}, we have
  $$
  \Bier(K)\vert_{I\sqcup\bar{J}} \simeq \Sigma^{\left\vert \eta \right\vert} \Lk(\eta,\Bier(K)\vert_{I\sqcup(\bar{J}\setminus\bar{\eta})}),
  $$
  and thus, it is also contractible, as desired.
\end{proof}

\begin{example}
  Recall the simplicial complex $K$ from Example~\ref{ex:3agt}.
  Let
  $$
  I=\{1,2,4\}, J_1=\{1,2\}, \text{ and } J_2=\{1,4\}.
  $$
  Note that~$K\vert_I$ has exactly two facets, namely~$\{1,2\}$ and~$\{4\}$.
  The link $\Lk(\{1,2\}, K\vert_I) = \{\emptyset\}$. 
  Hence, by Theorem~\ref{thm:4_main}, 
  $$
  \Bier(K)\vert_{I\sqcup \bar{J}_1}\simeq\Sigma^2\{\emptyset\} \cong S^1.
  $$
  
  Now consider the full subcomplex~$\Bier_{I \sqcup \bar{J}_2}$.
  We observe that~$\Lk(\{1\}, \Bier(K)\vert_{I\sqcup \{\bar{4}\}}) = \{2\}$.
  By Theorem~\ref{thm:4_main}, we obtain 
  $$
  \Bier(K)\vert_{I\sqcup\bar{J}_2}\simeq \Sigma \{2\} \cong D^1,
  $$
  where~$D^1$ denotes a closed interval.
  In fact, by Corollary~\ref{cor:case3_contractible}, $\Bier(K)\vert_{I\sqcup \bar{J}_2}$ is contractible.
  See Figure~\ref{figure:case3}.

  \begin{figure}
  \begin{tikzpicture}
        \node[left] at (6.6,1) {\large{$\Bier(K)$}};
        \filldraw[gray!20] (7.5, 0.5) -- (6.25, -1) -- (8, -1) -- cycle;
    \filldraw[gray!20] (7.5, 0.5) -- (8, -1) -- (8.75, -0.5) -- cycle;
    \filldraw[gray!20] (7.5, -2) -- (6.25, -1) -- (8, -1) -- cycle;
    \filldraw[gray!20] (7.5, -2) -- (8, -1) -- (8.5, -1.5) -- cycle;
    \filldraw[gray!20] (8.75, -0.5) -- (8, -1) -- (8.5, -1.5) -- cycle;
    \filldraw[gray!20] (7.5, 0.5) -- (6.25, -1) -- (7, -0.5) -- cycle;
    \filldraw[gray!20] (7.5, 0.5) -- (8.75, -0.5) -- (7, -0.5) -- cycle;
    \filldraw[gray!20] (7.5, -2) -- (7, -0.5) -- (6.25, -1) -- cycle;
    \filldraw[gray!20] (7.5, -2) -- (7, -0.5) -- (8.75, -0.5) -- cycle;
        \fill (7, -0.5) circle (1.5pt);
        \fill (6.25, -1) circle (1.5pt);
        \fill (8, -1) circle (1.5pt);
        \fill (8.75, -0.5) circle (1.5pt);
        \fill (7.5, 0.5) circle (1.5pt);
        \fill (7.5, -2) circle (1.5pt);
        \fill (8.5, -1.5) circle (1.5pt);
        \draw (6.5, 0) circle (1.5pt);

        \node at (7.5, 0.8) {1};
        \node at (7.2,-0.3) {2};
        \node at (6,-1) {3};
        \node at (8.7, -1.6) {4};
        \node at (7.5,-2.3) {$\bar{1}$};
        \node at (8.1, -0.7) {$\bar{2}$};
        \node at (8.9,-0.3) {$\bar{3}$};
        \node at (6.5, 0.3) {$\bar{4}$};

        \draw[dotted] (7.5, 0.5) -- (7, -0.5);
        \draw (7.5, 0.5) -- (6.25, -1);
        \draw (7.5, 0.5) -- (8,-1);
        \draw (7.5, 0.5) -- (8.75, -0.5);
        
        \draw[dotted] (7, -0.5) -- (6.25, -1);
        \draw (6.25, -1) -- (8, -1);
        \draw (8, -1) -- (8.75, -0.5);
        \draw[dotted] (8.75, -0.5) -- (7, -0.5);
        
        \draw[dotted] (7.5, -2) -- (7, -0.5);
        \draw (7.5, -2) -- (6.25, -1);
        \draw (7.5, -2) -- (8,-1);
        \draw[dotted] (7.5, -2) -- (8.75, -0.5);
        
        \draw (8.5, -1.5) -- (8.75, -0.5);
        \draw (8.5, -1.5) -- (8, -1);
        \draw (8.5, -1.5) -- (7.5, -2);

        \fill[fill=gray!50] (12.5,-2)--(13.5,0)--(14.5,-2);
        \node[left] at (12,1) {\large{$\Bier(K)\vert_{I \sqcup \bar{J}_1}$}};
        \foreach \x/\y/\num in {11.5/0/1, 12.5/-2/3, 10.5/-2/2} {
            \fill (\x, \y) circle (1.5pt);
        }
        \fill (13.5,0) circle (1.5pt);
        \fill (14.5,-2) circle (1.5pt);
        
        \node at (11.5,0.3) {1};
        \node at (10.5,-2.3) {2};
        \node at (12.5,-2.3) {$\bar{1}$};
        \node at (13.5,0.3) {$\bar{2}$};
        \node[below] at (14.5,-2) {$4$};
        
        \draw (11.5,0) -- (10.5,-2);
        \draw (12.5,-2) -- (10.5,-2);
        
        \draw (11.5,0)--(13.5,0);
        \draw (12.5,-2)--(13.5,0);
        \draw (13.5,0)--(14.5,-2);
        \draw (12.5,-2)--(14.5,-2);

        \node[left] at (17.6,1) {\large{$\Bier(K)\vert_{I \sqcup \bar{J}_2}$}};
        \foreach \x/\y/\num in {17/0/1, 17.5/-2/3, 16/-2/2} {
            \fill (\x, \y) circle (1.5pt);
        }
        \fill (18.5,-2) circle (1.5pt);
        \draw (15.5,0) circle (1.5pt);
        
        \node at (17,0.3) {1};
        \node at (16,-2.3) {2};
        \node at (17.5,-2.3) {$\bar{1}$};
        \node[below] at (18.5,-2) {$4$};
        \node[above] at (15.5,0) {$\bar{4}$};
        
        \draw (17,0) -- (16,-2);
        \draw (16,-2) -- (17.5,-2);
        
        \draw (17.5,-2)--(18.5,-2);
    \end{tikzpicture}
    \caption{Full subcomplexes $\Bier(K)_{I \sqcup \bar{J}}$}
    \label{figure:case3}
    \end{figure}
\end{example}

\section{Bigraded Betti numbers of Bier spheres}~\label{sec5}
Let $\Gamma$ be a simplicial complex on $[m] = \{1,\ldots,m\}$.
Denote by $\Q[v_1,\ldots,v_m]$ the graded polynomial ring in $v_1,\ldots, v_m$ over $\Q$ with $\deg v_i=2$ for each $i=1,\ldots,m$.
The \emph{Stanley-Reisner ring} (or the \emph{face ring}) $\Q(\Gamma)$ of $\Gamma$ over $\Q$ is defined to be $\Q[v_1,\ldots,v_m]/I_\Gamma$, where $I_\Gamma$ is the ideal generated by the monomials $v_{i_1}v_{i_2}\cdots v_{i_r}$ such that $\{i_1,i_2,\ldots,i_r\}\notin \Gamma$.
The Stanley-Reisner ring has been extensively studied in~~\cite{Reisner1976,Hochster1977,Stanley1996book}.

A \emph{finite free resolution} $[F:\phi]$ of $\Q(\Gamma)$ is an exact sequence
$$
0 \xrightarrow{     } F^{-h} \xrightarrow{\phi_h} F^{-h+1} \xrightarrow{\phi_{h-1}} \cdots \xrightarrow{\phi_2} F^{-1} \xrightarrow{\phi_1} F^0 \xrightarrow{\phi_0} \Q(\Gamma) \xrightarrow{} 0,
$$
where $F^{-i}$ is a finite free $\Q[v_1,\ldots,v_m]$-module and each~$\phi_i$ is degree-preserving.
If we take $F^0 = \Q[v_1,\ldots,v_m]$ as a $\Q[v_1,\ldots,v_m]$-module and $F^{-i}$ to be a $\Q[v_1,\ldots,v_m]$-module generated by a minimal basis of $\Ker(\phi_{i-1})$, then we get a \emph{minimal resolution} of $\Q(\Gamma)$.
Since $\Q(\Gamma)$ is graded, so are all $F^{-i}$, that is,
$$
F^{-i} = \bigoplus\limits_{j} F^{-i, 2j},
$$
where $F^{-i,2j}$ is the $2j$-th graded component of $F^{-i}$.
It is known~\cite{Buchstaber-Panov2002} that the \emph{Tor-algebra}
$$
\Tor_{\Q[v_1,\ldots,v_m]}^{-i,2j}(\Q(\Gamma),\Q)
$$
of $\Gamma$ with $(-i,2j)$-degree is isomorphic to $F^{-i,2j}$. 
We denote
$$
\beta^{-i, 2j}(\Gamma) = \dim_\Q F^{-i, 2j}
$$
and call it the \emph{bigraded Betti number} (or \emph{graded Betti number}) of $\Gamma$ in bidegree~$(-i, 2j)$.
Note that the bigraded Betti numbers do not depend on the choice of minimal resolution.

\begin{example}~\label{ex:minimal_resolution}
  Let $\Gamma$ be a simplicial complex on~$\{1,2,3,4\}$ defined by
  $$
  \Gamma = \{\emptyset, \{1\}, \{2\}, \{3\}, \{4\}, \{1,2\}, \{2,3\}, \{3,4\}, \{1,4\}\},
  $$
  which is the boundary of a square.
  The Stanley-Reisner ring $\Q(\Gamma)$ of $\Gamma$ is
  $$
  \frac{\Q[v_1, v_2, v_3, v_4]}{(v_1v_3, v_2v_4)},
  $$
  where $\deg v_i=2$ for $i=1,2,3,4$.
  For convenience, let~$R = \Q[v_1, v_2, v_3, v_4]$. 
  A minimal resolution is given by
  $$
  0\rightarrow R \langle w \rangle \xrightarrow{\phi_{2}}  R \langle u_{13}, u_{24} \rangle \xrightarrow{\phi_{1}} R \twoheadrightarrow \Q(\Gamma) \rightarrow 0,
  $$
  where the $R$-module homomorphisms are defined by
  \begin{align*}
    \phi_1 :  & u_{13} \mapsto v_1v_3, u_{24} \mapsto v_2v_4, \text{ and} \\
    \phi_2 : & w \mapsto v_2v_4 u_{13} - v_1v_3 u_{24}.
  \end{align*}
  We observe that~$\deg u_{13} = \deg u_{24} = 4$ and $\deg w = 8$.
  Thus, the bigraded Betti numbers of~$\Gamma$ are
  $$
  \beta^{-i,2j}(\Gamma)=
\begin{cases}
1, & (i,j)=(0,0),\\
2, & (i,j)=(1,2),\\
1, & (i,j)=(2,4),\\
0, & \text{otherwise}.
\end{cases}
  $$
\end{example}

As demonstrated in Example~\ref{ex:minimal_resolution}, the computation of a minimal resolution for obtaining the bigraded Betti numbers is highly nontrivial.
When one more vertex is added to $\Gamma$, even a ghost one, the length of the minimal resolution increases.
Adding a ghost vertex to~$\Gamma$ leaves all nonzero bigraded Betti numbers unchanged, though some bigraded Betti numbers that were zero may become nonzero.
For instance, let us consider the simplicial complex~$\Gamma'$ obtained from $\Gamma$ in Example~\ref{ex:minimal_resolution} by adding two ghost vertices.
A computation of a minimal resolution of~$\Gamma'$ yields exactly six additional nonzero bigraded Betti numbers as follows
\begin{equation}\label{eq:Gamma'}
\begin{aligned}
\beta^{0,0}(\Gamma')&=1, \qquad&\beta^{-1,2}(\Gamma') = 2,\\ \beta^{-1,4}(\Gamma')&= 2, \qquad &\beta^{-2,4}(\Gamma')=1,\\ \beta^{-2,8}(\Gamma')& = 1, \qquad &\beta^{-2,6}(\Gamma')=4,\\ & & \beta^{-3,8}(\Gamma')=2,\\ && \beta^{-3,10}(\Gamma') = 2, &\mbox{ and}\\ &&\beta^{-4,12}(\Gamma') = 1.
\end{aligned}
\end{equation}

To compute the bigraded Betti numbers, a fundamental tool is Hochster's formula~\cite{Hochster1977}, given by
$$
\beta^{-i,2j}(\Gamma;\Q) = \sum_{\substack{I \subset V(\Gamma)\\ \left\vert I \right\vert = j}} \widetilde{\beta}^{j-i-1}(\Gamma\vert_I;\Q),
$$
where $V(\Gamma)$ denotes the vertex set of $\Gamma$.
Therefore, computing the bigraded Betti numbers of $\Gamma$ via Hochster's formula requires a complete understanding of all full subcomplexes of $\Gamma$.
With the homotopy types of all full subcomplexes of Bier spheres at hand, we now describe the bigraded Betti numbers of a Bier sphere via Hochster's formula.

For a simplicial complex~$K$ on~$[m]$, we define collections~$\cF^+_k$ and~$\cF^-_k$ by
\begin{align*}
  &\cF^+_k := \{\sigma\subseteq [m] : \left\vert \sigma \right\vert = \frac{k}{2}, \sigma\in K, \bar{\sigma} \in \hat{K} \}, \mbox{ and }\\
  &\cF^-_k := \{\tau\subseteq [m] : \left\vert \tau \right\vert = \frac{k}{2}, \tau\not\in K, \bar{\tau}\not\in \hat{K}\}. 
  \end{align*}

\begin{theorem}~\label{thm:bigraded}
  Let $K$ be a simplicial complex on $[m] = \{1,\ldots,m\}$.
  The bigraded Betti number of the Bier sphere~$\Bier(K)$ in bidegree~$(-i, 2j)$ is given by
  \begin{align*}
                              \delta_{2i,j} \left\vert \cF^+_j\right\vert + \delta_{2i-2,j}\left\vert \cF^-_j \right\vert
                              &
                              + \sum_{\substack{J\subsetneq I \subseteq [m] \\ \left\vert I \right\vert + \left\vert J \right\vert = j \\ J \in K}} \widetilde{\beta}^{\left\vert I \right\vert-i-1}(\Lk(J, K\vert_I))
                             &
                              + \sum_{\substack{I\subsetneq J \subseteq [m] \\ \left\vert I \right\vert + \left\vert J \right\vert = j \\ \bar{I} \in \hat{K}}} \widetilde{\beta}^{\left\vert J \right\vert-i-1}(\Lk(\bar{I}, \hat{K}\vert_{\bar{J}}))
                              ,
  \end{align*}
  where $\delta_{i,j}$ is Kronecker delta.
\end{theorem}

\begin{proof}
Let~$I$ and~$J$ be subsets of~$[m]$ such that~$\left\vert I \right\vert + \left\vert J \right\vert = j$.

We first consider the case when $I=J$.
In this case, we have~$\left\vert I \right\vert = j/2$.
By~\eqref{thm:main1} and~\eqref{thm:main2} in Theorem~\ref{thm:main}, we observe that
      \begin{align}\label{eq:bithm1}
      \widetilde{\beta}^{j-i-1} \big(\Bier(K)\vert_{I\sqcup\bar{I}}\big)&=
      \begin{cases}
        \widetilde{\beta}^{j-i-1}(S^{\left\vert I \right\vert - 1}) =\delta_{j,2i}, & \text{if }  I \in \cF^+_{j} , \\[4pt]
        \widetilde{\beta}^{j-i-1}(S^{\left\vert I \right\vert - 2}) =\delta_{j,2i-2}, & \text{if }  I \in \cF^-_{j}.
      \end{cases}
       \end{align}
       
  Next, we consider the case where $I$ and $J$ are properly comparable under inclusion.
  Without loss of generality, we assume that~$J \subsetneq I$ and~$J \in K$.
  Then
  \begin{equation}\label{eq:bithm2}
    \begin{aligned}
    \widetilde{\beta}^{j-i-1} \big(\Bier(K)\vert_{I\sqcup\bar{J}}\big)&= \widetilde{\beta}^{j-i-1} \big(\Sigma^{\left\vert J\right\vert}\Lk(J, K\vert_{I})\big) & \text{(By \eqref{thm:main3} in Theorem~\ref{thm:main})} \\
     & = \widetilde{\beta}^{j-i-1-\left\vert J\right\vert} \big(\Lk(J, K\vert_{I})\big) &\\
     & = \widetilde{\beta}^{\left\vert I\right\vert-i-1} \big(\Lk(J, K\vert_{I})\big) & \text{(as $\left\vert I \right\vert + \left\vert J \right\vert = j$)}.
  \end{aligned}
  \end{equation} 
    Similarly, by \eqref{thm:main4} in Theorem~\ref{thm:main}, if~$I \subsetneq J$ and~$\bar{I} \in \hat{K}$, we obtain
     \begin{equation}\label{eq:bithm3}
      \widetilde{\beta}^{j-i-1} \big(\Bier(K)\vert_{I\sqcup\bar{J}}\big)= \widetilde{\beta}^{\left\vert J\right\vert-i-1} \big(\Lk(\bar{I}, \hat{K}\vert_{\bar{J}})\big).
     \end{equation}
  
  From Hochster's formula, we obtain
  $$
  \beta^{-i,2j}(\Bier(K)) = \sum_{\substack{I,J\subseteq[m]\\ \left\vert I \right\vert + \left\vert J \right\vert = j}}\widetilde{\beta}^{j-i-1}(\Bier(K)\vert_{I\sqcup\bar{J}}).
  $$
  Applying Theorem~\ref{thm:main} together with \eqref{eq:bithm1},  
\eqref{eq:bithm2}, and \eqref{eq:bithm3} completes the proof.
\end{proof}

Using Theorem~\ref{thm:bigraded}, the bigraded Betti numbers of Bier spheres are obtained without the computation of minimal resolution.

\begin{example} \label{Example:interval_with_ghost}
  Let $K$ be a simplicial complex on $\{1,2,3\}$ in which $3$ is a ghost vertex, and defined by
  $$
    K = \{\emptyset, \{1\}, \{2\}, \{1,2\}\}.
  $$
  Then $K$ is self-dual, and the Bier sphere $\Bier(K)$ of $K$ is the boundary of square with two ghost vertices.
  Note that~$\Bier(K)$ is (simplicial) isomorphic to~$\Gamma'$ in~\eqref{eq:Gamma'}.

One can check that
\[
\begin{aligned}
&\begin{minipage}{0.42\textwidth}
\[
\delta_{2i,j}\,|\cF^+_j|=
\begin{cases}
1, & (i,j) = (0,0),\\
2, & (i,j) = (1,2),\\
1, & (i,j) = (2,4),\\
0, & \text{otherwise},
\end{cases}
\]
\end{minipage}
\quad\mbox{and}\quad
\begin{minipage}{0.42\textwidth}
\[
\delta_{2i-2,j}\,|\cF^-_j|=
\begin{cases}
1, & (i,j) = (2,2),\\
2, & (i,j) = (3,4),\\
1, & (i,j) = (4,6),\\
0, & \text{otherwise}.
\end{cases}
\]
\end{minipage}
\end{aligned}
\]

\begin{table}
    \centering
    \renewcommand{\arraystretch}{1.3}
    \setlength{\tabcolsep}{5pt}
    \begin{tabular}{|c|c|c|c|c|}
        \hline
        $j$ & $I$ & $J$ & $\Lk(J, K\vert_{I})$ & $\widetilde{\beta}^{k}$ \\
        \hline
    
  & \begin{tabular}{@{}c@{}} $\{1\}$\end{tabular}
     & \begin{tabular}{@{}c@{}} $\emptyset$\end{tabular}
    & \begin{tikzpicture}[scale=0.5]
        \fill (0,0) circle (3pt);
      \end{tikzpicture}
    & $\widetilde{\beta}^k (\Lk(J, K\vert_{I})) = 0$\\
         \cline{2-5}
          \parbox[c][4.6\baselineskip][c]{1.6em}{\centering 1}

   & \begin{tabular}{@{}c@{}} $\{2\}$\end{tabular}
     & \begin{tabular}{@{}c@{}} $\emptyset$\end{tabular}
    & \begin{tikzpicture}[scale=0.5]
        \fill (0,0) circle (3pt);
      \end{tikzpicture}
    & $\widetilde{\beta}^k (\Lk(J, K\vert_{I})) = 0$\\
 \cline{2-5}

   & \begin{tabular}{@{}c@{}} $\{3\}$\end{tabular}
     & \begin{tabular}{@{}c@{}} $\emptyset$\end{tabular}
    & \begin{tikzpicture}[scale=0.5]
        \draw (0,0) circle (3pt);
      \end{tikzpicture}
    & $
        \widetilde{\beta}^{k}(\Lk(J, K\vert_{I}))=
        \begin{cases}
            1 & \mbox{if } k=-1\\
            0 & \mbox{otherwise}
        \end{cases}
      $\\
\hline
   & \begin{tabular}{@{}c@{}} $\{1, 2\}$\end{tabular}
     & \begin{tabular}{@{}c@{}} $\emptyset$\end{tabular}
    & \begin{tikzpicture}[scale=0.5]
        \fill (-1,0) circle (3pt);
        \fill (1,0) circle (3pt);
        \draw (-1,0)--(1,0);
      \end{tikzpicture}
    & $\widetilde{\beta}^k (\Lk(J, K\vert_{I})) = 0$
 \\
 \cline{2-5}
 \parbox[c][4.6\baselineskip][c]{1.6em}{\centering 2}
 & \begin{tabular}{@{}c@{}} $\{1, 3\}$\\ $\{2, 3\}$\end{tabular}
    & \begin{tabular}{@{}c@{}} $\emptyset$\end{tabular}
    & \begin{tikzpicture}[scale=0.5]
        \fill (0,0) circle (3pt);
      \end{tikzpicture}
    & $\widetilde{\beta}^k (\Lk(J, K\vert_{I})) = 0$
 \\
 \hline
    & \begin{tabular}{@{}c@{}} $\{1, 2\}$\\ \\\end{tabular}
    & \begin{tabular}{@{}c@{}} $\{1\}$ \\ $\{2\}$\end{tabular}
    & \begin{tikzpicture}[scale=0.5]
        \fill (0,0) circle (3pt);
      \end{tikzpicture}
    & $\widetilde{\beta}^k (\Lk(J, K\vert_{I})) = 0$ \\
    \cline{2-5} 
    \parbox[c][4.6\baselineskip][c]{1.6em}{\centering 3}
    & \begin{tabular}{@{}c@{}}$\{1, 3\}$\\$\{2, 3\}$\end{tabular}
    & \begin{tabular}{@{}c@{}}$\{1\}$\\$\{2\}$\end{tabular}
    & \begin{tikzpicture}[scale=0.5]
        \draw (0,0) circle (3pt);
      \end{tikzpicture}
    & $
        \widetilde{\beta}^{k}(\Lk(J, K\vert_{I}))=
        \begin{cases}
            1 & \mbox{if } k=-1\\
            0 & \mbox{otherwise}
        \end{cases}
      $
    \\
    \cline{2-5}
 &
 \begin{tabular}{@{}c@{}}$\{1, 3\}$\\$\{2, 3\}$\end{tabular}
 &
 $\{3\}$
 &
 $J\not\in K\vert_I$
 & $\times$
 \\
\hline
$4$ &
\begin{tabular}[c]{@{}l@{}}
$\{1,2,3\}$
\end{tabular}
&
\begin{tabular}[c]{@{}l@{}}
$\{1\}$\\
$\{2\}$
\end{tabular}
&
\begin{tabular}[c]{@{}c@{}}
\begin{tikzpicture}[scale=0.5]
  \fill (0,0)circle (3pt);
\end{tikzpicture}
\end{tabular}
&
$\widetilde{\beta}^k (\Lk(J, K\vert_{I})) = 0$ \\
\cline{3-5}
&
&
$\{3\}$
&
$J\not\in K\vert_I$
&
$\times$\\
\hline
$5$ &
\begin{tabular}[c]{@{}l@{}}
$\{1,2,3\}$
\end{tabular}
&
\begin{tabular}[c]{@{}l@{}}
$\{1,2\}$
\end{tabular}
&
\begin{tabular}[c]{@{}c@{}}
\begin{tikzpicture}[scale=0.5]
  \draw (0,0)circle (3pt);
\end{tikzpicture}
\end{tabular}
&
$
    \widetilde{\beta}^{k}(\Lk(J, K\vert_{I}))=\begin{cases}
        1 & \mbox{if \;} k=-1\\
        0 & \mbox{otherwise}
    \end{cases}
$
\\
\cline{3-5}
&
&
\begin{tabular}[c]{@{}l@{}}
$\{1,3\}$ \\ $\{2,3\}$
\end{tabular}
&
$J\not\in K\vert_{I}$
&
$\times$\\
\hline
\end{tabular}
\caption{The links of $K$ and their reduced Betti numbers}
\label{table:links}
\end{table}
From Table~\ref{table:links}, we also obtain the reduced Betti number of the links needed for the third and fourth terms in Theorem~\ref{thm:bigraded}.
Therefore, we provide the Betti table of $\beta^{-i,2j}(\Bier(K))$, see Table~\ref{table:betti_table}.

\begin{table}[h]
\centering
\scalebox{1.2}{
\begin{tabular}{c| ccccccc}
\toprule
$i\backslash j$ & 0 & 1 & 2 & 3 & 4 & 5 & 6 \\
\midrule
0 & 1 & $\cdot$ & $\cdot$ & $\cdot$ & $\cdot$ & $\cdot$ & $\cdot$\\
1 & $\cdot$ & 2 & 2 & $\cdot$ & $\cdot$ & $\cdot$ & $\cdot$\\
2 & $\cdot$ & $\cdot$ & 1 & 4 & 1 & $\cdot$ & $\cdot$\\
3 & $\cdot$ & $\cdot$ & $\cdot$ & $\cdot$ & 2& 2&$\cdot$\\
4 & $\cdot$ & $\cdot$ & $\cdot$ & $\cdot$ & $\cdot$ & $\cdot$ &1\\ 
\bottomrule
\end{tabular}
}
\caption{Betti table of $\beta^{-i,2j}(\Bier(K))$}
\label{table:betti_table}
\end{table}
These may now be directly compared with the formula in~\eqref{eq:Gamma'}.
\end{example}

\begin{example}[Skeletons of simplices] \label{Example:skeletons_of_simplex}
Let $m \geq 3$ be a positive integer and let $r$ be an integer with $-1\leq r \leq \lfloor \frac{m-1}{2} \rfloor -1$.
For the $(m-1)$-simplex~$\Delta^{m-1}$ on $[m]$, the \red{$r$}-skeleton~$\Delta^{m-1}_{\red{r}}$ of~$\Delta^{m-1}$ is defined as
$$
    \Delta^{m-1}_{\red{r}} = \{\sigma\subset[m] \colon \left\vert \sigma \right\vert \leq \red{r}+1\},
$$
and its Alexander dual coincides with $\Delta^{m-1}_{m-\red{r}-3}$, where the $(-1)$-skeleton is the empty set.
  
  Let $I,J$ be subsets of $[m]$ and $j=\left\vert I \right\vert + \left\vert J \right\vert$.
  Since $-1\leq r \leq \lfloor \frac{m-1}{2} \rfloor -1$, one observes that $I\in\Delta^{m-1}_{r}$ implies that $I\in\Delta^{m-1}_{m-r-3}$; conversely, $I\notin\Delta^{m-1}_{m-r-3}$ implies that $I\notin\Delta^{m-1}_{r}$.
  Therefore, $\left\vert \cF^+_j\right\vert$ and $\left\vert \cF^-_j\right\vert$ for $\Bier(\Delta^{m-1}_r)$ are given by
  \[
\begin{aligned}
&\begin{minipage}{0.42\textwidth}
\[
|\cF^+_j|=
\begin{cases}
\binom{m}{j/2}, & \mbox{if } j/2\leq r+1\\
0, & \mbox{otherwise},
\end{cases}
\]
\end{minipage}
\quad\mbox{and}\quad
\begin{minipage}{0.42\textwidth}
\[
|\cF^-_j|=
\begin{cases}
\binom{m}{j/2}, & \mbox{if } j/2\geq m-r-2\\
0, & \text{otherwise}.
\end{cases}
\]
\end{minipage}
\end{aligned}
\]
  
  Let us consider the case $J\subsetneq I$.
  By Theorem~\ref{thm:bigraded}, it suffices to compute the reduced Betti numbers of $\Lk(J,\Delta^{m-1}_{r}\vert_I)$.
  Since $J\in \Delta^{m-1}_{r}\vert_I$, we have $\left\vert J \right\vert \leq r+1$.
  When $I \in \Delta^{m-1}_{r}$, $\Delta^{m-1}_{r}\vert_I$ is the $(\left\vert I \right\vert-1)$-simplex and hence contractible; therefore, it suffices to consider the nontrivial case $\left\vert I \right\vert > r+1$.
  In this case, $\Delta^{m-1}_{r}\vert_I$ coincides with $\Delta^{\left\vert I \right\vert-1}_{r}$, and consequently $\Lk(J, \Delta^{m-1}_{r}\vert_I) = \Delta^{\left\vert I \right\vert-\left\vert J \right\vert-1}_{r-\left\vert J \right\vert}.$
  It is well-known~\cite{Hatcher2002} that the $n$-skeleton of the $k$-simplex has the homotopy type of a wedge sum of $\binom{k}{n+1}$ $n$-spheres for integers $n$ and $k$. 
  Applying this fact, 
  \begin{equation}~\label{eq:homotopy_of_skeleton}
  \Lk(J,\Delta^{m-1}_{r}\vert_I) \simeq \bigvee_{\binom{\left\vert I \right\vert-\left\vert J \right\vert-1}{r-\left\vert J \right\vert+1}} S^{r-\left\vert J \right\vert}.
  \end{equation}
  Therefore, 
  $$
  \widetilde{\beta}^{\left\vert I \right\vert - i - 1}(\Lk(J,\Delta^{m-1}_{r}\vert_I ))=
  \begin{cases}
    \binom{\left\vert I \right\vert-\left\vert J \right\vert-1}{r-\left\vert J \right\vert+1}, & \mbox{if } i=j-r-1, \\
    0, & \mbox{otherwise}.
  \end{cases}
  $$
  
  Similarly, when $I\subsetneq J$, we have $\left\vert I \right\vert \leq m-r-2<\left\vert J \right\vert$, and hence $\Lk(\bar{I}, \Delta^{m-1}_{m-r-3}\vert_{\bar{J}}) = \Delta^{\left\vert J \right\vert-\left\vert I \right\vert-1}_{m-r-3-\left\vert I \right\vert}$.
  Applying the same homotopy type computation as~\eqref{eq:homotopy_of_skeleton}, we obtain
  $$
    \widetilde{\beta}^{\left\vert J \right\vert-i-1}(\Lk(\bar{I},\Delta^{m-1}_{m-r-3}\vert_{\bar{J}}))=
  \begin{cases}
    \binom{\left\vert J \right\vert-\left\vert I \right\vert-1}{m-r-3-\left\vert J \right\vert+1}, & \mbox{if } i=j-m+r+2, \\
    0, & \mbox{otherwise}.
  \end{cases}
  $$  
  
  Using Theorem~\ref{thm:bigraded}, we obtain the following description of the bigraded Betti numbers of $\Bier(\Delta^{m-1}_{r})$.
In particular, when $m=2k+3$ and $r=k$ for an integer $-1\leq k$, $\Delta^{m-1}_{r}$ is self-dual, and thus the cases~\eqref{thm:main3} and~\eqref{thm:main4} coincide.
Hence, in this situation, we obtain the following.
\begin{align*}
  \beta^{-i,2j}\!\left(\Bier(\Delta^{2k+2}_{k})\right)
  = 
  \begin{cases}
    \displaystyle\binom{m}{i}, 
      &\text{if } i \le k+1,\; 2i = j, \\[10pt]
    \displaystyle\binom{m}{i-1},
    &\text{if } i \ge k+2, \; 2i=j+2, \\[10pt]
    2 \displaystyle
      \sum_{\substack{J \subsetneq I \subset [2k+3] \\[1pt]
        |J| \le k+1 < |I|}}
      \binom{|I| - |J| - 1}{\,k - |J| + 1\,},
      & \text{if } i = j - k - 1, \\[25pt]
    0, 
      & \text{otherwise}.
  \end{cases}
\end{align*}
On the other hand, for general $m$ and $-1\leq r \leq \lfloor \frac{m-1}{2} \rfloor -1$, we obtain
\begin{align*}
  \beta^{-i,2j}\!\left(\Bier(\Delta^{m-1}_{r})\right)
  =
  \begin{cases}
    \displaystyle\binom{m}{i}, 
      &\text{if } i \le r+1,\; 2i = j \\[10pt]
    \displaystyle\binom{m}{i-1},&\mbox{if } i \ge m-r,\; 2i = j+2, \\[10pt]
    \displaystyle
    \sum_{\substack{J \subsetneq I \subset [m] \\[1pt]
      |J| \le r+1 < |I|}}
      \binom{|I| - |J| - 1}{\,r - |J| + 1\,},
      & \text{if } i = j - r - 1, \\[15pt]
    \displaystyle
    \sum_{\substack{I \subsetneq J \subset [m] \\[1pt]
      |I| \le m-r-2 < |J|}}
      \binom{|J| - |I| - 1}{\,m - r - 2 - |I|\,},
      & \text{if } i = j - m + r + 2, \\[25pt]
    0,
      & \text{otherwise}.
  \end{cases}
\end{align*}

It should be noted that the computation above recovers the result of Heudtlass and Katth\"an in~\cite[Proposition~$4.6$]{Heudtlass-Katthan2012}.

\end{example}

\section{Real toric manifolds associated with Bier spheres}\label{sec:topo}
Unless otherwise stated, throughout this section, we assume that $K \neq 2^{[m]}$ is a simplicial complex on~$[m] = \{1,2,\ldots,m\}$.
We focus on the real toric manifold~$\realbier$ associated with the Bier sphere~$\Bier(K)$ of~$K$ and the canonical characteristic matrix~$\Lambda_m$, defined in~\eqref{matrix}.

We provide an explicit description of the integral cohomology group of $\realbier$ in terms of the faces of $K$. 
For each $0 \leq k \leq \lfloor \frac{m}{2} \rfloor$, we consider the two types of collections of subsets of $[m]$ that are either contained in both $K$ and $\hat{K}$, or in neither of them, and define these two collections as follows 
\begin{align*}  
    \cI_{2k}(K) &:= \{ I \subset [m] \colon |I|=2k, I \in K, \text{ and } \bar{I} \in \hat{K} \}, \text{ and}\\    
    \cI_{2k-1}(K) &:= \{ I \subset [m] \colon |I|=2k, I \not\in K, \text{ and } \bar{I} \not\in \hat{K} \}.  
\end{align*}  
From Theorem~\ref{sec3:thm}, it follows that, for each~$0 \leq i \leq m-1$,
\begin{equation}\label{sec4:eq1}
  \bigoplus_{0 \leq k \leq \lfloor \frac{m}{2} \rfloor}\bigoplus_{I \in \binom{[m]}{2k}}\widetilde{H}^{i-1}(\Bier(K)\vert_{I \sqcup \bar{I}};\Z) = \bigoplus_{I \in \cI_i(K)}\widetilde{H}^{i-1}(\Bier(K)\vert_{I \sqcup \bar{I}};\Z) = \Z^{\left\vert \cI_{i}(K) \right\vert}.
\end{equation}
Combining \eqref{sec4:eq1} together with Theorem~\ref{CaiChoi} and~\eqref{eq:row-2subset}, we have the following theorem.

\begin{theorem}\label{sec4:thm1}
    For each $0 \leq i \leq m-1$, the $i$th (rational) Betti number of~$\realbier$ is equal to~$\left\vert \cI_{i}(K) \right\vert$.
    Furthermore, $H^\ast(\realbier;\Z)$ is $p$-torsion free for all $p \geq 3$.
\end{theorem}

Let~$M^\R$ be an $n$-dimensional real toric manifold associated with~$(\Gamma,\Lambda)$.
If there is a function~$\cG \colon \row \Lambda \to \{0,1,\ldots,n\}$ such that, for each~$\omega \in \row \Lambda$,
$$
\tilde{\beta}^{i-1}(\Gamma_\omega) = 0 \quad \text{for all } i \neq \cG(\omega),
$$
the (rational) homology group~$H^\ast(M^\R)$ is said to be \emph{homological concentrated} in~$\cG$, following~\cite{Choi-Yoon2026}.

Note that, from~\eqref{eq:row-2subset}, each element of the row space~$\row \Lambda_m$ is naturally identified with a subset~$I \subset [m]$ of even cardinality.
We present a function
$$
\cG_{m} \colon \bigoplus_{0 \leq k \leq \lfloor \frac{m}{2} \rfloor} \binom{[m]}{2k} \longrightarrow \{0,1,\ldots,m-2\}
$$
defined by
$$
\cG_{m}(\omega)=
\begin{cases}
i, & \text{if } I \in \cI_i(K),\\[2pt]
0, & \text{otherwise}.
\end{cases}
$$
For a simplicial complex~$K$ on~$[m]$, the homology group~$H^\ast(\realbier)$ is homological concentrated in~$\cG_{m}$.
Furthermore, $I \in \cI_i$ contributes not only to~$\beta_i(\realbier)$ but to exactly one contribution.

\begin{example}[Example~\ref{Example:skeletons_of_simplex} revisited]

    Consider the $r$-skeleton~$\Delta^{m-1}_r$ of the $(m-1)$-simplex on~$[m]$.
    It follows that
    $$
    \left\vert \cI_{2k}(\Delta^{m-1}_r) \right\vert=
    \begin{cases}
      \binom{m}{2k}, & \text{if } 2k \leq \min\{r+1,m-2-r\} \\
      0, & \text{otherwise},
    \end{cases}
    $$
    and 
    $$
    \left\vert \cI_{2k-1}(\Delta^{m-1}_r) \right\vert=
    \begin{cases}
      \binom{m}{2k}, & \text{if } 2k > \max\{r+1,m-2-r\} \\
      0, & \text{otherwise}.
    \end{cases}
    $$
    In particular, when $r = m-2$, $\Bier(\Delta^{m-1}_r)$ is the boundary of the $(m-1)$-simplex, and the corresponding real toric manifold is known to be the real projective space~$\RP^{m-1}$ of dimension~$m-1$.
    The above computation coincides with the known result that $\RP^{m-1}$ is rationally acyclic if $m-1$ is even, and a rational homology sphere if $m-1$ is odd.
\end{example}

For a simplicial complex~$\Gamma$ on a finite set~$S$, we introduce two important combinatorial invariants; the \emph{$f$-vector} and the \emph{$h$-vector} of a simplicial complex $\Gamma$ with $r$ vertices. 
The integral vector $(f_0,f_1,\ldots,f_{r-1})$ is called the $f$-vector of~$\Gamma$, denoted by~$f(\Gamma)$, where $f_i$ is the number of $i$-dimensional faces of~$\Gamma$ for each $0 \leq i \leq n-1$.
Similarly, the integral vector $(h_0,h_1,\ldots,h_{r})$ is called the $h$-vector of~$\Gamma$, denoted by~$h(\Gamma)$, and is defined by the polynomial equation
$$
h_0t^r + \cdots + h_{r-1}t + h_r = (t-1)^r + f_0(t-1)^{r-1} + \cdots + f_{r-1},
$$
where $f_{-1} = 1$.
In particular, if $\Gamma$ is a simplicial sphere, then
$$
h_i = h_{r-i} \text{ for each } 0 \leq i \leq r, \text{ and } h_0 \leq h_1 \leq \cdots \leq h_{\lfloor \frac{r}{2} \rfloor},
$$
which is part of the famous \emph{$g$-theorem}.
Refer~\cite{Stanley1980,Adiprasito2018,Karu-Xiao2023}.

\begin{corollary}\label{sec4:coro1}
    For each $0 \leq k \leq \lfloor \frac{m}{2} \rfloor$,
    $$
    h_{2k}(\Bier(K))-h_{2k-1}(\Bier(K)) = \beta^{2k} - \beta^{2k-1},
    $$
    where $\beta^{i}$ is the $i$th (rational) Betti number of $\realbier$. 
\end{corollary}

\begin{proof}
  From~\cite[Theorem~5.2]{Bjorner-Paffenholz-Sjostrand-Ziegler2005}, we obtain that, for each $1 \leq k \leq \lfloor \frac{m}{2} \rfloor$,
  $$
  h_{2k}-h_{2k-1} = f_{2k-1}(K) - f_{m-2k-1}(K),
  $$
  where the $f$-vector of $K$ is $(f_0(K),\ldots,f_{m-1}(K))$.
  Observe that
  \begin{align*}
    f_{2k-1}(K) - f_{m-2k-1}(K) = & \left\vert \{I \subset [m] \colon \left\vert I \right\vert = 2k, I \in K\}\right\vert - \left\vert \{I \subset [m] \colon \left\vert I \right\vert = m-2k, I \in K\}\right\vert\\
    = & \left\vert \{I \subset [m] \colon \left\vert I \right\vert = 2k, I \in K\}\right\vert - \left\vert \{I \subset [m] \colon \left\vert I \right\vert = 2k, \bar{I} \notin \hat{K}\}\right\vert\\
     = &  \Bigl(\left\vert \cI_{2k}(K) \right\vert + \left\vert \{I \subset [m] \colon \left\vert I \right\vert = 2k, I \in K, \bar{I} \notin \hat{K}\}\right\vert\Bigl)\\
      & - \Bigl(\left\vert \{I \subset [m] \colon \left\vert I \right\vert = 2k, I \in K, \bar{I} \notin \hat{K}\}\right\vert + \left\vert \cI_{2k-1}(K)\right\vert\Bigl)\\
      = & \left\vert \cI_{2k}(K) \right\vert -  \left\vert \cI_{2k-1}(K)\right\vert \\
      = & \beta^{2k} - \beta^{2k-1}.
  \end{align*}
  This completes the proof.
\end{proof}

\begin{remark}
  Since $\Bier(K)$ is a simplicial sphere, by combining the $g$-theorem with Corollary~\ref{sec4:coro1}, we observe that
    $$
    \begin{cases}
      \beta^{2k-1} < \beta^{2k}, & \text{if } 0 \leq k \leq \lfloor \frac{m}{4} \rfloor, \\
      \beta^{2k-1} > \beta^{2k}, & \text{if } \lfloor \frac{m}{4} \rfloor < k \leq \lfloor \frac{m}{2} \rfloor.
    \end{cases}
    $$
However, there is no clear pattern between $\beta^{2k}$ and $\beta^{2k+1}$.
\end{remark}

Combinatorial descriptions of the multiplicative structure for the cohomology ring of real toric spaces have been studied in~\cite{Choi-Park2017_multiplicative}.
Some explicit cases such as those in~\cite{Choi-Yoon2023,Yoon2025} are successfully described. 
However, providing a fully explicit combinatorial description of the multiplicative structure of cohomology ring of general real toric spaces is quite challenging.
In the remainder of this section, we discuss the multiplicative structure of the rational cohomology ring~$H^\ast(\realbier)$.

From~\eqref{sec3:eq2}, the Bier sphere $\Bier(K\vert_I)$ is regarded as a subcomplex of~$\Bier(K)\vert_{I\sqcup \bar{I}}$.
Then, the map
\begin{equation}\label{sec5:eq1}
   \varphi \colon \bigoplus_{i=0}^{m-1} \Q \left\langle I \subset [m] \colon I \in \cI_i \right\rangle \to \bigoplus_{i=0}^{m-1}\bigoplus_{I \in \cI_{i}(K)}\widetilde{H}^{i-1}(\Bier(K)\vert_{I\sqcup\bar{I}})
\end{equation}
defined by
$$
\varphi(I) =\begin{cases}
           \bigl[\Bier(K)\vert_{I \sqcup \bar{I}}\bigl], & \text{if } I = \{i_1,\ldots,i_{2k}\} \in \cI_{2k}(K), \\
           \bigl[ \Bier(K\vert_I) \bigl], & \text{if } I \in \cI_{2k-1}(K),
         \end{cases}
$$
is a well-defined $\Q$-module homomorphism.
From Theorem~\ref{sec3:thm}, it follows that $\varphi$ is a graded $\Q$-module isomorphism.
Combining Theorem~\ref{ChoiPark} with \eqref{sec4:eq1}, there is a graded $\Q$-algebra isomorphism
\begin{equation}\label{sec5:eq2}
\bigoplus_{i=0}^{m-1}\bigoplus_{I \in \cI_{i}(K)}\widetilde{H}^{i-1}(\Bier(K)\vert_{I\sqcup\bar{I}}) \to H^\ast(\realbier).
\end{equation}
By the composition of maps in \eqref{sec5:eq1} and \eqref{sec5:eq2}, we have
$$
\Phi \colon \bigoplus_{i=0}^{m-1} \Q \left\langle I \subset [m] \colon I \in \cI_i \right\rangle \to H^\ast(\realbier). 
$$ 
Note that each~$I \in \cI_i(K)$ serves as a cohomology generator of degree~$i$.

Let $\smile$ be the cup product in the rational cohomology ring~$H^\ast(\realbier)$.
Recall the symmetric difference~$I \triangle J = (I \cup J)\setminus(I \cap J)$ of~$I$ and~$J$.
From Theorem~\ref{ChoiPark}, the multiplicative structure on
$$
\bigoplus_{i=0}^{m-1}\bigoplus_{I \in \cI_{i}(K)}\widetilde{H}^{i-1}(\Bier(K)\vert_{I\sqcup\bar{I}})
$$
is given by the canonical maps
$$
\smile \colon \widetilde{H}^{i-1}(\Bier(K)\vert_{I\sqcup\bar{I}}) \otimes \widetilde{H}^{j-1}(\Bier(K)\vert_{J\sqcup\bar{J}}) \to \widetilde{H}^{i+j-1}(\Bier(K)\vert_{{(I\triangle J)}\sqcup(\overline{I \triangle J})}),
$$
for $I \in \cI_{i}(K)$ and $J \in \cI_{j}(K)$, which are induced by simplicial maps
$$
\Bier(K)\vert_{{(I\triangle J)}\sqcup(\overline{I \triangle J})} \hookrightarrow \Bier(K)\vert_{I\sqcup\bar{I}} \ast \Bier(K)\vert_{J\sqcup\bar{J}},
$$
where $\ast$ is the simplicial join.

\begin{theorem}\label{thm:multi}
For $I \in \cI_{i}(K)$ and $J \in \cI_{j}(K)$, the following holds.
\begin{enumerate}
\item\label{oddeven} If $i$ and $j$ have different parity, then $\Phi(I) \smile \Phi(J) = 0$, whenever $I \cap J \neq \emptyset$,
\item\label{oddodd} if both $i$ and $j$ are odd, then $\Phi(I) \smile \Phi(J) = 0$, whenever $\left\vert I \cap J \right\vert \neq 1$, and
\item\label{eveneven} if both $i$ and $j$ are even,
      $$
    \Phi(I) \smile \Phi(J) =  
\begin{cases}
\Phi(I \sqcup J), \text{ up to a sign}, & \text{if } I \sqcup J \in \cI_{i+j}(K),\\
0, & \text{otherwise}.
\end{cases}  
    $$
    \end{enumerate}
\end{theorem}
\begin{proof}
    To establish the proof, we consider two distinct cases; both $i$ and $j$ are odd, or at least one of them is even.

    First, we assume that both $i$ and $j$ are odd.
    Then, we have $\left\vert I \right\vert = i+1$ and $\left\vert J \right\vert = j+1$.
    It follows that $\left\vert I \triangle J \right\vert = i+j+2$ if $I \cap J = \emptyset$, and $\left\vert I \triangle J \right\vert \leq i+j-2$ if $\left\vert I \cap J \right\vert \geq 2$.
    By Theorem~\ref{sec3:thm}, since $I \triangle J \notin \cI_{i+j}(K)$, we obtain
    $$
    \widetilde{H}^{i+j-1}(\Bier(K)\vert_{{(I\triangle J)}\sqcup(\overline{I \triangle J})}) = 0.
    $$
    Hence, the product~$\Phi(I) \smile \Phi(J)$ of~$I$ and~$J$ is trivial, which verifies that the second case in~\eqref{oddodd} holds.

    Next, we consider the case where at least one of~$i$ and~$j$ is even.
    We observe that~$\left\vert I \right\vert +\left\vert J \right\vert = i+j$ or~$i+j+1$.
    If $I \cap J \neq \emptyset$, then $\left\vert I \triangle J \right\vert \leq i+j-1$, which implies that~$I \triangle J \notin \cI_{i+j}(K)$.
    By the same argument as above, the product~$\Phi(I) \smile \Phi(J)$ vanishes in this case as well.
    Thus, the case~\eqref{oddeven} holds, and the case~\eqref{eveneven} holds when $I \cup J \neq \emptyset$.
    
    We now assume that both $i$ and $j$ are even, with $I \sqcup J = \cI_{i+j}(K)$.
    Let $I = \{i_1,\ldots,i_{2k}\}$ and $J = \{j_1,\ldots,j_{2\ell}\}$.
    Then, we obtain that
    \begin{align*}
  \Bier(K)\vert_{{(I\sqcup J)}\sqcup(\overline{I \sqcup J})} & = \bigl\{\{i_1\},\{\bar{i}_1\}\bigl\}\ast \cdots \ast \bigl\{\{i_{2k}\},\{\bar{i}_{2k}\}\bigl\} \ast \bigl\{\{j_1\},\{\bar{j}_2\}\bigl\}\ast \cdots \ast \bigl\{\{j_{2\ell}\},\{\bar{j}_{2\ell}\}\bigl\}\\
       & = \Bier(K)\vert_{I\sqcup\bar{I}} \ast \Bier(K)\vert_{J\sqcup\bar{J}}.
    \end{align*}
    Therefore, $\Phi(I) \smile \Phi(J) = \Phi(I \sqcup J)$, up to a sign.
\end{proof}

\begin{remark}\label{remark:orien}
    By fixing an orientation of~$\Bier(K)$, one can explicitly determine the sign through a detailed computation.
    For the non-trivial case of Theorem~\ref{thm:multi}, the sign is given by the permutation that distinguishes the orientations of~$\Bier(K)\vert_{I\sqcup\bar{I}} \ast \Bier(K)\vert_{J\sqcup\bar{J}}$ and~$\Bier(K)\vert_{{(I\sqcup J)}\sqcup(\overline{I \sqcup J})}$.
\end{remark}

\begin{conjecture}\label{conj}
    For each $I \in \cI_{i}(K)$ and $J \in \cI_{j}(K)$,
    \begin{enumerate}
      \item $\Phi(I) \smile \Phi(J) = \Phi(I \sqcup J)$, up to a sign, if $i$ and $j$ have different parity with $I \cap J = \emptyset$, and
      \item if both $i$ and $j$ are odd with $\left\vert I \cap J \right\vert = 1$,
          $$
          \Phi(I) \smile \Phi(J) = \begin{cases}
                                     \Phi(I \triangle J), \text{ up to a sign,} & \text{if } I \triangle J \in \cI_{i+j}(K), \\
                                     0, & \text{if } I \triangle J \notin \cI_{i+j}(K).
                                   \end{cases}
          $$
    \end{enumerate}
Also, in non-trivial cases, the sign can be computed by referring to their orientations.
\end{conjecture}

After performing calculations in small dimensions, the authors verified Conjecture~\ref{conj}, but in general, it remains unclear whether the conjecture holds.
If confirmed, it would yield a full description of the multiplicative structure of the rational cohomology ring~$H^\ast(\realbier)$.

\begin{example}
    Consider a simplicial complex~$K$ on~$S = \{1,2,3,4,5,6,7\}$ with the following facets
    $$
    \{2, 3\}, \{1, 3, 6, 7\}, \{1, 5, 7\}, \{1, 2, 4, 5, 6\}.
    $$
    Thus, the maximal faces of the Alexander dual $\hat{K}$ of $K$ are
    $$
\{\bar{1}, \bar{2}, \bar{4}, \bar{6}, \bar{7}\}, \{\bar{1}, \bar{2}, \bar{3}, \bar{4}\}, \{\bar{1}, \bar{3}, \bar{4}, \bar{5}, \bar{6}\}, \{\bar{1}, \bar{2}, \bar{5}, \bar{6}, \bar{7}\}, \{\bar{1}, \bar{2}, \bar{3}, \bar{5}, \bar{6}\}, \{\bar{1}, \bar{4}, \bar{5}, \bar{7}\}, \{\bar{4}, \bar{5}, \bar{6}, \bar{7}\}.
$$
By a direct computation, we obtain the following collections:
    \begin{align*}
    \cI_{0}(K) =& \bigl\{\emptyset\bigl\}, \\
  \cI_{2}(K) =&  \bigl\{I \subset S \colon \left\vert I \right\vert = 2\bigl\} \setminus \bigl\{ \{3, 4\}, \{3, 5\}, \{2, 7\}, \{4, 7\}, \{3, 7\} \bigl\}, \\
  \cI_{3}(K) = &  \bigl\{\{1, 2, 3, 6\}, \{1, 2, 3, 7\}, \{1, 2, 4, 7\}, \{1, 2, 5, 7\}, \{1, 3, 4, 6\}, \{1, 3, 4, 7\}, \\
& \{1, 3, 5, 6\}, \{1, 3, 5, 7\}, \{1, 4, 5, 6\}, \{2, 3, 4, 6\}, \{2, 3, 4, 7\}, \{2, 3, 5, 7\}\bigl\},\\
  \cI_{4}(K) =&  \bigl\{\{1,2,4,6\},\{1,2,5,6\},\{1,4,5,6\}\bigl\},\\
  \cI_{5}(K) =& \bigl\{I \subset S \colon \left\vert I \right\vert = 6\bigl\},\text{ and}\\
    \cI_{i}(K) = & \ \emptyset, \text{ if either $i=1$ or $i \geq 6$}.
\end{align*}
In conclusion, by Theorem~\ref{sec4:thm1}, the (rational) Betti numbers $\beta^i$ of $\realbier$ are as follows
$$
\begin{array}{c|c|c|c|c|c|c|c}
i & 0 & 1 & 2 & 3 & 4 & 5 & \geq 6\\ \hline
\beta^i & 1 & 0 & 16 & 12 & 3 & 7 & 0
\end{array}.
$$

By Theorem~\ref{thm:multi}, one can partially compute the cup product~$\smile$ in~$H^{\ast}(\realbier)$.
For instance, we have the following examples.
\begin{enumerate}
  \item $\Phi(\{1,2,3,6\}) \smile \Phi(\{1,4\}) = 0$, and
  \item $\Phi(\{1,2\}) \smile \Phi(\{4,6\}) = \Phi(\{1,2,4,6\})$, up to a sign.
\end{enumerate}
\end{example}

\providecommand{\bysame}{\leavevmode\hbox to3em{\hrulefill}\thinspace}
\providecommand{\MR}{\relax\ifhmode\unskip\space\fi MR }
\providecommand{\MRhref}[2]{%
  \href{http://www.ams.org/mathscinet-getitem?mr=#1}{#2}
}
\providecommand{\href}[2]{#2}

\end{document}